\documentclass[a4paper, oneside]{amsart}

\usepackage[top=2.5cm,bottom=2.5cm,left=3.5cm,right=2.5cm]{geometry}
\calclayout

\usepackage[utf8]{inputenc}
\usepackage{graphicx}
\usepackage{amsmath}
\usepackage{amssymb}
\usepackage{amsthm}
\usepackage{dsfont}
\usepackage{color}
\usepackage{enumerate}
\usepackage[colorlinks=true]{hyperref}

\newtheorem{theorem}{Theorem}

\newtheorem{proposition}[theorem]{Proposition}
\newtheorem{criterion}[theorem]{Criterion}
\newtheorem{lemma}[theorem]{Lemma}

\theoremstyle{remark}
\newtheorem{remark}{Remark}

\graphicspath{ {./} }

\begin{document}

\title{$1/f^{3/2}$ noise in heavy traffic M/M/1 queue}

\author{T.~Vanderbruggen}
\address{\parbox{\linewidth}{Koheron\\3 rue du Sous Marin Venus\\56100 Lorient\\France\\}}
\email{thomas.vanderbruggen@koheron.com}

\date{\today}
\let\thefootnote\relax\footnotetext{Licence \href{https://creativecommons.org/licenses/by/4.0/}{CC-BY 4.0}}

\begin{abstract}
We study the power spectral density of continuous time Markov chains and explicit its relationship with the eigenstructure of the infinitesimal generator. 
This result helps us understand the dynamics of the number of customers for a M/M/1 queuing process in the heavy traffic regime.
Closed-form relations for the power law scalings associated to the eigenspectrum of the M/M/1 queue generator 
are obtained, providing a detailed description of the power spectral density structure, which is shown to exhibit a $1/f^{3/2}$ noise.
We confirm this result by numerical simulation.
We also show that a continuous time random walk on a ring exhibits similar behavior with the same scaling exponent.
It is remarkable that a complex behavior such as $1/f$ noise can emerge from the M/M/1 queue, which is the "simplest" queuing model.
\end{abstract}

\maketitle

\section{Introduction}

Flicker noise also named $1/f$ noise \cite{Milotti2002, Ward2007} describes a stochastic process whose power spectral density (PSD)
exhibits a power law behavior $1/f^{\alpha}$ over a certain range of frequencies $f$, usually dominating at low frequency.
The cases $\alpha = 0$ and $\alpha = 2$ correspond to white noise and Brownian motion noise, respectively.
Both are well understood and are simply related to each other by derivation / integration - The Brownian motion being the integral of the white noise.
Hence a $1/f^{\alpha}$ noise is usually defined for $\alpha$ in the range $0 < \alpha < 2$, excluding the trivial cases $\alpha = 0$ and $\alpha = 2$.
In other words a $1/f^{\alpha}$ noise cannot be obtained directly from simple transformation on the white noise,
that is the reason why its mathematical description is challenging.

On the experimental side the flicker noise appears in a wide range of systems. 
For example: heart rate \cite{Berndt1999}, human cognition \cite{Gilden1995},
music \cite{Voss1978}, internet traffic \cite{Csabai1994, Takayasu2005}, electronic devices \cite{vanDerZiel1988},
magnetic field fluctuations of the Sun \cite{Greco2016} and Saturn \cite{Hadid2015}.
However, it is still an open question whether the widespread existence of this behavior is explained by a common underlying mathematical mechanism.

Several mechanisms have been proposed to generate a $1/f^{\alpha}$ noise, and a detailed presentation can be found in \cite{Ward2007}.
The two main mathematical frameworks are the superposition of Poisson processes with widespread relaxation rates
\cite{Walma1980, Yamamoto, Jiang2003, Erland2007, Erland2011}, and non-linear stochastic differential equations \cite{Ruseckas2010, Ruseckas2014}.
The superposition of relaxation rates, whose resulting PSD is a sum of Lorentzians, is an interesting candidate,
especially because reversible homogeneous continuous time  Markov processes exhibit such a spectrum, as shown in \cite{Jiang2003, Erland2007}.
Here, we redemonstrate this result and show that the cut-off frequencies are the eigenvalues of the infinitesimal generator and the amplitude
of each Lorentzian is given by the coupling of the state vector to the eigenvectors.
In particular, we show that if the eigenvalues and the eigenvectors follow appropriate power laws a $1/f^{\alpha}$ noise scaling can be achieved. 
However, to the extent of our knowledge, mostly \textit{ad hoc} Markov chain with tailored eigenspectrum distributions have been introduced
to show that such models can generate $1/f$ noise.
Here we show that $1/f$ noise appears in a "classical" Markov model, namely the M/M/1 queue.

The M/M/1 queue is a birth-death process with constant transition rates, making it the \textit{simplest} birth-death processes
(beyond pure birth and pure death processes).
A major advantage of the M/M/1 queue simplicity is that many closed-form expressions are known for various queue metrics
(see for example \cite{Haviv2013}).
We apply the previously introduced result on the PSD of reversible Markov chains to the infinitesimal generator
of a M/M/1 queue in the heavy traffic regime.
Scaling laws for the eigenvalues and eigenvectors of the generator are obtained from its diagonalization and used
to show that the PSD exhibits a $1/f^{3/2}$ noise, which is confirmed by numerical simulation.
Finally, we consider a continuous time random walk on a ring and show that its generator eigenstructure is similar to the M/M/1 queue generator,
and therefore it also exhibits a $1/f^{3/2}$ noise as was already noticed \cite{Erland2007, Yadav2021}.

\section{Definitions and notations}

\subsection{Continuous time Markov chains}

We consider homogeneous continuous time Markov chains (HCTMC) on a finite coutable state space of size $n$.
A HCTMC $X=\left\{ X_{t}, t \geq 0 \right\}$ is characterized by its transition matrix $\mathbf{P} \left( \tau \right)$
with elements $\mathbf{P}_{ij} \left( \tau \right) = P \left( X_{\tau} = j | X_{0} = i \right)$.
It is a stochastic matrix, because $\sum_{i=1}^{n} \mathbf{P}_{ij} \left( \tau \right) = 1$,
which obeys to the Kolmogorov forward equation: 
\begin{equation}
\frac{d}{d\tau} \mathbf{P} \left( \tau \right) + \mathbf{G} \mathbf{P} \left( \tau \right) = \mathbf{0},
\label{eq:kolmogorov_forward}
\end{equation}
where $\mathbf{G}$ is called the \textit{infinitesimal generator} of the HCTMC.
Since a continuous time Markov chain over a finite state space is regular, there is exactly one solution to this differential equation.
Given the initial condition $\mathbf{P} (0) = \mathds{1}$ must be satisfied, the solution of this differential equation is 
\begin{equation}
\mathbf{P} \left( \tau \right) = e^{ - \mathbf{G} \tau }.
\label{eq:sol_gen}
\end{equation}
The generator satisfies $\sum_{i} \mathbf{G}_{ij} = 0$ and all the off-diagonal elements must be negative. 

The long term behavior of the Markov process is characterized by the stationary probability vector $\boldsymbol\pi = (\pi_{i})_{1 \leq i \leq n}$,
where $\pi_{i}$ is the probability to be in the state $i$ of the chain after an infinite time.
We further assume the chain is \textit{irreducible} so that the stationary distribution is unique, because the state space is finite. 
From Eq.~(\ref{eq:kolmogorov_forward}), it satisfies $\boldsymbol\pi \mathbf{G} = \mathbf{0}$ with $\| \boldsymbol\pi \|_{1} = 1$.

Finally, we consider a \textit{reversible} chain, that is a chain which satisfies the detailed balance equations 
$\pi_{i} \mathbf{G}_{ij} = \pi_{j} \mathbf{G}_{ji}$.
A reversible Markov process has the desirable property to be stationary, and therefore the power spectral density of this process exists.
Because the matrix $\mathbf{D}^{1/2}\mathbf{G}\mathbf{D}^{-1/2}$, where $\mathbf{D} = \mathrm{diag}(\boldsymbol\pi)$,
is symmetric and similar to $\mathbf{G}$,
the generator of a reversible chain is diagonalizable with real positive eigenvalues \cite{3157533}.
Note that birth-death processes, and in particular the M/M/1 queue, are reversible (see for example \cite{Haviv2013} 8.4.3).

\subsection{Scalar product}

In the context of Markov chains the scalar product whose metric is the stationary probability occurs naturally.
To each state $i$ of the chain we associate a real value $x_{i}$, and we note $\mathbf{x} = (x_{i})_{1 \leq i \leq n} \in \mathbb{R}^{n}$.
For a HCTMC $X$ with stationary probability $\boldsymbol\pi$, the space $\mathbb{R}^{n}$ is equipped with the scalar product
\begin{equation}
\left\langle \mathbf{u}, \mathbf{v} \right\rangle _{\boldsymbol\pi} \triangleq \sum_{i=1}^{n} \pi_{i} u_{i} v_{i}.
\end{equation}
The all-ones vector is noted $\mathbf{1}$, it verifies
$\left\langle \mathbf{1}, \mathbf{1} \right\rangle _{\boldsymbol\pi} = \left\| \boldsymbol\pi \right\|_{1} = 1$.
This scalar product relates to statistical quantities such as the average or second order moment:
\begin{align}
\left\langle \mathbf{x}, \mathbf{1} \right\rangle _{\boldsymbol\pi} &= \sum_{i=1}^{n} \pi_{i} x_{i} = \left\langle x \right\rangle, \\
\left\langle \mathbf{x}, \mathbf{x} \right\rangle _{\boldsymbol\pi} &= \sum_{i=1}^{n} \pi_{i} x_{i}^{2} = \left\langle x^{2} \right\rangle.
\end{align}

\section{Power spectral density of a Markov process}

We re-establish a known result that the PSD of a reversible Markov process is a superposition of Lorentzians \cite{Jiang2003, Erland2007}.
The derivation explicits the relationship with the fundamental matrix of the Markov chain as an intermediate result.
We first derive the expression for the autocorrelation function of a stationary Markov process,
and, combining this result with the Green-Kubo relation, we show the relation between the diffusion coefficient of the process
and the fundamental matrix of the Markov chain.
Interestingly, the Green-Kubo relation is somehow a particular case of the Wiener-Kinchine theorem
(the diffusion coefficient is the PSD at zero frequency), and we show that the PSD of the Markov chain
relates to a generalized version of the fundamental matrix.
Then we explicit the relationship between the generalized fundamental matrix and the infinitesimal generator,
and we show that the PSD of a reversible process is a sum of Lorentzians.
Finally, we establish a criterion relating power law scalings of the generator eigenstructure and $1/f$ noise scaling.

\subsection{Autocorrelation function}

Here we consider the centered process $\delta X_{t} \triangleq X_{t} - \left\langle X_{t} \right\rangle$,
in other word we remove the DC component from the power spectrum.
The autocorrelation function is defined as
\begin{equation}
\mathcal{C}_{X} (\tau) \triangleq \lim_{t \rightarrow \infty} \left\langle \delta X_{t + \tau} \delta X_{t} \right\rangle.
\end{equation}
In App.~\ref{app:autocorrelation} we prove that the autocorrelation function of a HCTMC is a quadratic form:
\begin{proposition}
Let $X$ be a stationary and irreducible HCTMC and let $\mathbf{x}$ be a real-valued vector over the (finite) state space of the chain,
then the autocorrelation of $X$ is
\begin{equation}
\mathcal{C}_{X} (\tau) = \left\langle \mathbf{x}, \left[ \mathbf{P} (\tau) - \mathbf{P} (\infty) \right] \mathbf{x} \right\rangle _{\boldsymbol\pi}.
\end{equation}
\label{prop:autocorrelation}
\end{proposition}
Where we defined the limiting transition matrix $\mathbf{P} (\infty)$,
whose lines are the stationary probability vector $\mathbf{P}_{ij} (\infty) \triangleq \lim_{\tau \rightarrow \infty} \mathbf{P}_{ij} (\tau) = \pi_{j}$.

Note that this result and the next ones assume a finite state space with implies that the stationary population is unique,
but the results would hold if an unique distribution exists, whether the state space is finite or not.

\subsection{Fundamental matrix and Green-Kubo relation}

The \textit{fundamental matrix}, also called deviation matrix, of a HCTMC is defined as
\begin{equation}
\mathbf{Z} \triangleq \int_{0}^{\infty} \left[ \mathbf{P} (\tau) - \mathbf{P} (\infty) \right] \; d\tau.
\label{eq:fund_matrix}
\end{equation}
It can be related to many properties of the Markov chain, such as 
variance \cite{Whitt1992}, first-passage time analysis \cite{Yao1985}, speed of convergence to the stationary distribution \cite{CoolenSchrijner2002}.
Interestingly, we see here that, by the Green-Kubo relation, it is also related to the \textit{diffusion coefficient} of the process,
and more generally to some transport coefficient of the process.

The diffusion coefficient of the process $X$ is $D_{X} = D_{\partial_{t} Y} \triangleq \lim_{t \rightarrow \infty} \sigma_{Y}^{2} (t) / 2 t$,
where $\sigma_{Y}^{2} (t)$ is the variance of the process $Y$ which is the integral of the process $X$ ($X = \partial_{t} Y$) \cite{Geisel1988}.
The diffusion coefficient is related to the autocorrelation function by the Green-Kubo relation \cite{dalibard1985, Geisel1988}:
\begin{equation}
D_{X} = \int_{0}^{\infty} \mathcal{C}_{X} (\tau) d\tau.
\label{prop:green-kubo-def}
\end{equation}
Using the result Prop.~\ref{prop:autocorrelation} for the autocorrelation function, it follows
\begin{equation}
D_{X} = \left\langle \mathbf{x}, \left\{ \int_{0}^{\infty}  \left[ \mathbf{P} (\tau) - \mathbf{P} (\infty) \right] d\tau \right\} \mathbf{x} \right\rangle _{\boldsymbol\pi},
\end{equation}
and we observe that the Green-Kubo relation for the diffusion coefficient of a stationary Markov process is a quadratic form of the fundamental matrix:
\begin{proposition}[Green-Kubo relation]
Let $X$ be a stationary and irreducible HCTMC and $\mathbf{x}$ be a real-valued vector indexed over the (finite) state space of the chain,
then the diffusion coefficient of $X$ is
\begin{equation}
D_{X} = \left\langle \mathbf{x}, \mathbf{Z} \mathbf{x} \right\rangle _{\boldsymbol\pi}.
\label{eq:green-kubo}
\end{equation}
\label{prop:green-kubo}
\end{proposition}

\subsection{Power spectral density}

The Wiener-Kinchine theorem states that the power spectral density is the Fourier transform of the autocorrelation function.
More precisely, if $X$ is a wide-sense stationary process then the power spectral density of $X$ is
\begin{equation}
S_{X} ( \omega ) = \int_{0}^{\infty} \mathcal{C}_{X} (\tau) \cos (\omega \tau) d\tau.
\end{equation}
We note that this is related to the Green-Kubo relation Eq.~(\ref{prop:green-kubo-def}) because $D_{X} = S_{X}(0)$,
and the power spectral density also describes out-of-equilibrium fluctuations around the stationary state of the system.

Because a stationary process is wide-sense stationary, applying the Wiener-Kinchine theorem to the result Prop.~\ref{prop:autocorrelation}
for the autocorrelation function yield:
\begin{equation}
S_{X} ( \omega ) = \left\langle \mathbf{x}, \left\{ \int_{0}^{\infty}  \left[ \mathbf{P} (\tau) - \mathbf{P} (\infty) \right] 
                                                      \cos (\omega \tau) d\tau \right\} \mathbf{x} \right\rangle _{\boldsymbol\pi}.
\end{equation}
Now we define a generalized version of the fundamental matrix Eq.~(\ref{eq:fund_matrix}):
\begin{equation}
\mathbf{Z} (\omega) \triangleq \int_{0}^{\infty} \left[ \mathbf{P} (\tau) - \mathbf{P} (\infty) \right] \cos (\omega \tau) \; d\tau,
\end{equation}
where the fundamental matrix is retrieved at the zero frequency: $\mathbf{Z} = \mathbf{Z}(0)$.
Hence the Wiener-Kinchine theorem for a stationary Markov chain states that the PSD is a quadratic form with matrix $\mathbf{Z} (\omega)$:
\begin{proposition}[Wiener-Kinchine]
Let $X$ be a stationary and irreducible HCTMC and $\mathbf{x}$ be a real-valued vector indexed over the (finite) state space of the chain,
then the power spectral density of $X$ is
\begin{equation}
S_{X} (\omega) = \left\langle \mathbf{x}, \mathbf{Z} (\omega) \mathbf{x} \right\rangle _{\boldsymbol\pi}.
\end{equation}
\label{prop:quad_form_psd}
\end{proposition}

The following property, proven in App.~\ref{app:proof_gen_fund_mat} not assuming $\mathbf{G}$ to be diagonalizable,
relates the fundamental matrix to the infinitesimal generator:
\begin{proposition}
The generalized fundamental matrix of an irreducible HCTMC over a finite state space with generator $\mathbf{G}$ is
\begin{equation}
\mathbf{Z} (\omega) = \left( \mathbf{G} + \omega^{2} \mathbf{G}^{\sharp} \right)^{\sharp},
\end{equation}
where $\mathbf{A}^{\sharp}$ is the group inverse of the matrix $\mathbf{A}$.
\label{prop:gen_fund_mat}
\end{proposition}
In particular, $\mathbf{Z} = \mathbf{Z}(0) = \mathbf{G}^{\sharp}$ and we recover the known fact that the
fundamental matrix is the group inverse of the generator (see \cite{CoolenSchrijner2002} for example).
The Green-Kubo relation Eq.~(\ref{eq:green-kubo}) is thus
$D_{X} = \left\langle \mathbf{x}, \mathbf{G}^{\sharp} \mathbf{x} \right\rangle _{\boldsymbol\pi}$,
as was observed in \cite{Vanderbruggen2013}.

\subsection{Sum of Lorentzians}

In the case of a reversible chain, which is diagonalizable with real positive eigenvalues,
we can combine the results of Prop.~\ref{prop:quad_form_psd} and Prop.~\ref{prop:gen_fund_mat}
to obtain an explicit expression relating the power spectral density to eigenspectrum structure of the generator
(see details in App.~\ref{app:proof_sum_lorentzian}):

\begin{theorem}
Let $X$ be an irreducible and reversible HCTMC over a finite state space of size $n$.
If the non-zero eigenvalues of the infinitesimal generator of $X$ are $\{ \omega_{k} \}_{1 \leq k \leq n - 1}$, 
and $\boldsymbol\Pi_{k}$ is the spectral projector associated to the eigenvalue $\omega_{k}$, then
the power spectral density of $X$ is
\begin{equation}
S_{X} (\omega) = \sum_{k = 1}^{n-1} \gamma_{k}^{2} \frac{\omega_{k}}{\omega_{k}^{2} + \omega^{2}},
\label{eq:sum_lorentzians}
\end{equation}
where $\gamma_{k}^{2} \triangleq \left\langle \mathbf{x}, \boldsymbol\Pi_{k} \mathbf{x} \right\rangle _{\boldsymbol\pi}$.
\label{th:sum_lorentzians}
\end{theorem}

We observe that the power spectral density of a Markov chain is a \textit{sum of Lorentzian}.
Similar results are obtained for reversible Markov chain \cite{Jiang2003, Erland2007, Erland2011},
nonlinear stochastic differential equations \cite{Ruseckas2010}
or the spectral density of the laser field near threshold \cite[18.7.3]{mandel1995}.
Note that since the process is \textit{reversible} every non-zero eigenvalue of $\mathbf{G}$ is real and strictly positive,
so that $S_{X} (\omega) \geq 0$ as expected.


The coefficient $\gamma_{k}^{2}$ corresponds to the coupling strength of the $k$th eigenvector to the state vector $\mathbf{x}$.
In App.~\ref{app:graph_fourier_transform}, we show that for symmetric infinitesimal generators $\gamma_{k}$ is related to the
\textit{graph Fourier transform} of the signal $\mathbf{x}$ over the vertices of the (undirected) graph associated to the Markov chain.

\begin{remark}[Normalization of $\gamma_{k}$ coefficients]
Since the spectral projectors satisfies the resolution of unity $\sum_{k = 1}^{n-1} \boldsymbol\Pi_{k} = \mathds{1}$,
the coupling strength has the normalization:
\begin{equation}
\sum_{k = 1}^{n-1} \gamma_{k}^{2} = \left\langle \mathbf{x}, \mathbf{x} \right\rangle _{\boldsymbol\pi} = \left\langle x^{2} \right\rangle.
\label{eq:normalization_gammak}
\end{equation}
\end{remark}

\begin{remark}[Normalization of $S_{X} (\omega)$]
The overall energy of the spectrum is, using Eq.~(\ref{eq:normalization_gammak}):
\begin{equation}
\int_{0}^{\infty} S_{X} (\omega) \; d\omega = \sum_{k = 1}^{n-1} \gamma_{k}^{2} \int_{0}^{\infty} \frac{\omega_{k}}{\omega_{k}^{2} + \omega^{2}} \; d\omega
= \frac{\pi}{2} \sum_{k = 1}^{n-1} \gamma_{k}^{2} = \frac{\pi}{2} \left\langle x^{2} \right\rangle.
\end{equation}
Therefore if we desire to satisfy $\int_{0}^{\infty} S_{X} (\omega) \; d\omega = \left\langle x^{2} \right\rangle$, we must include a normalization
factor $2 / \pi$ to the result Eq.~(\ref{eq:sum_lorentzians}).
\end{remark}

\subsection{Lorentzian power spectrum}

In the specific case where the eigenspectrum of $\mathbf{G}$ is degenerate ($\omega_{k} = \bar{\omega}$ for all $k > 0$),
which occurs in particular for a two states chain, then the power spectrum is Lorentzian:
\begin{equation}
S_{X} (\omega) = \frac{\left\langle x^{2} \right\rangle}{\bar{\omega}} \frac{1}{1 + (\omega / \bar{\omega} )^{2}}
= \frac{D_{X}}{1 + (\omega / \bar{\omega} )^{2}}.
\end{equation}

\begin{remark}
Another, more contrived, way to obtain a Lorentzian spectrum is for $\mathbf{x}$ to be an eigenvector of $\mathbf{G}$.
\end{remark}

As an example of a process with a highly degenerate spectrum we consider the \textit{continuous time random walk on star graph}.
It is a process where one transition from a central state to a peripherical one with rate $\lambda$, 
and back to the central state from the periphery with rate $\mu$.
The transition graph is presented in Fig.~\ref{fig:graph_rw_star}.

\begin{figure}[!h]
\begin{center}
\includegraphics[width=5cm,keepaspectratio]{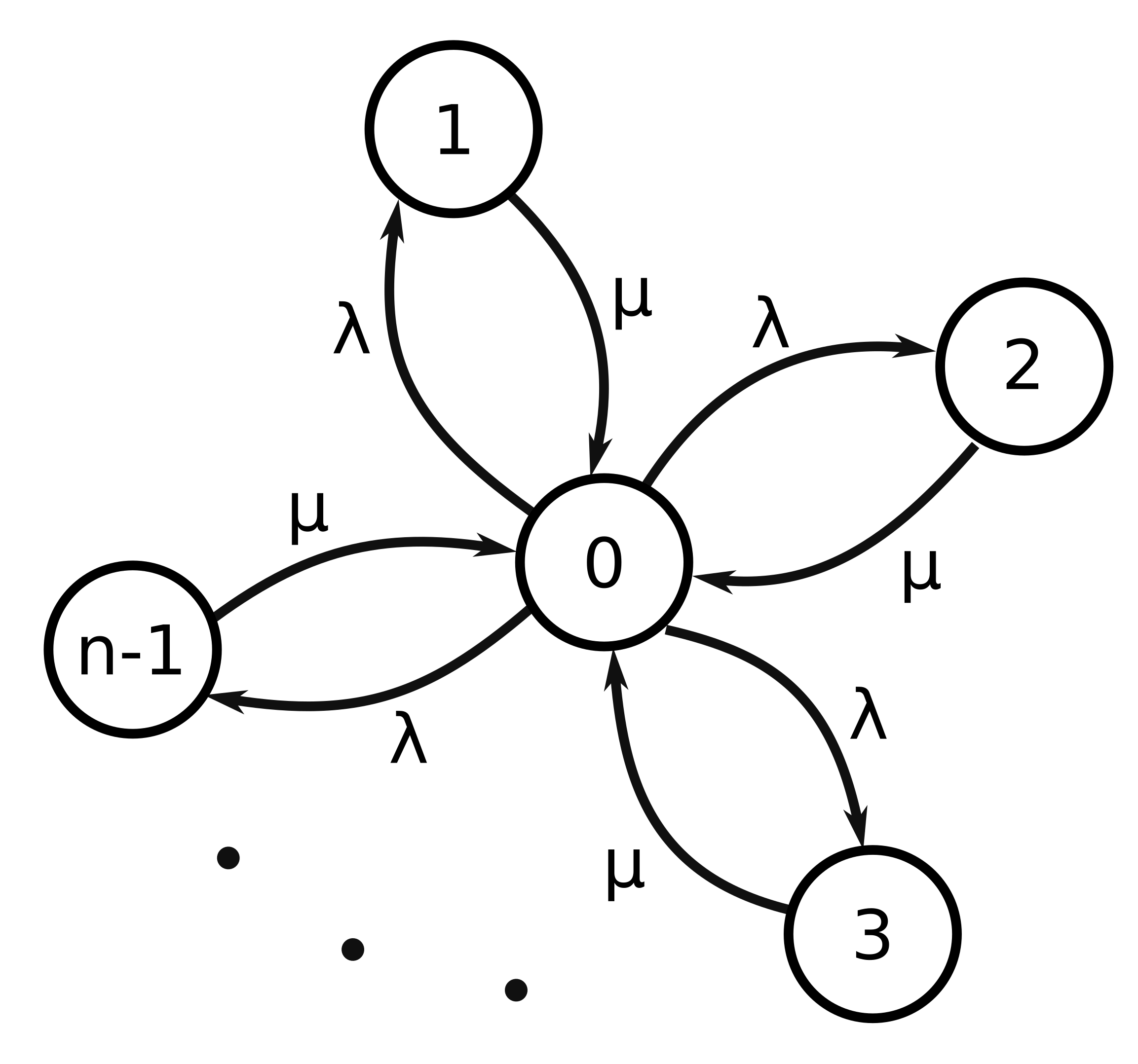}
\caption{Transitions for the random walk on a star graph.}
\label{fig:graph_rw_star}
\end{center}
\end{figure}

The infinitesimal generator of this process is an arrowhead matrix:
\begin{equation}
\mathbf{G} =
\begin{pmatrix}
& (n - 1) \lambda & -\lambda  & -\lambda & \dots  \\
& -\mu            & \mu       &          &        \\
& -\mu            &           & \mu      &        \\
& \vdots          &           &          & \ddots \\
\end{pmatrix}.
\label{eq:inf_gen_star_graph_process}
\end{equation}
We further assume $\lambda = \mu = 1$ and the generator is the Laplacian matrix of an undirected star graph.
The eigenvalues are $\omega_{0} = 0$, $\omega_{1} = \omega_{2} = \dots = \omega_{n-2} = 1$ and $\omega_{n-1} = n$ (see for example \cite{Brouwer2012} 15.3.3 (ii)).
Therefore, excluding the zero eigenvalue, the generator only has 2 eigenvalues $1$ and $n$ 
so the spectrum is strongly degenerate and the power spectral density is nearly Lorentzian, as we can observe on the simulation Fig.~\ref{fig:simul_rw_star}.

\begin{remark}
In the special case where $n = 1$ one obtain the telegraph process which gives rise to random telegraph noise (also known as burst noise or popcorn noise),
whose power spectrum is well known to be Lorentzian \cite{Leyris2006}.
\end{remark}

\begin{figure}[!h]
\begin{center}
\includegraphics[width=15cm,keepaspectratio]{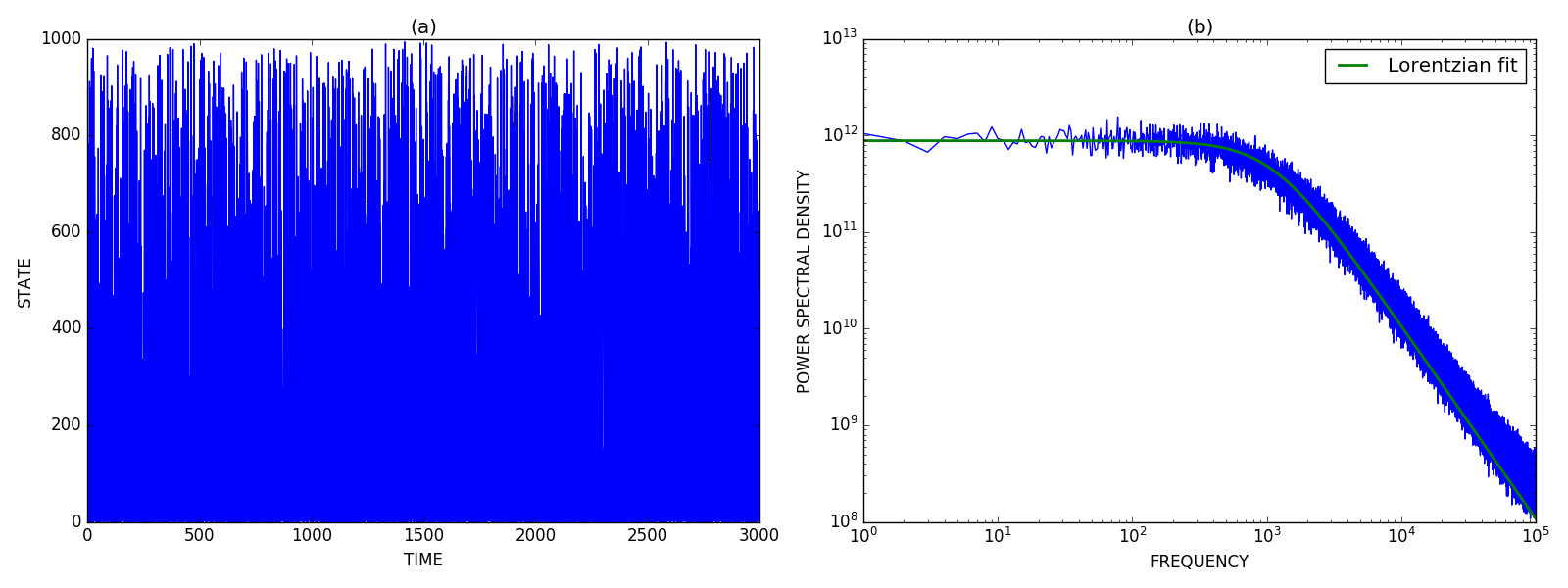}
\caption{Simulation of the continuous time random walk on a star graph of size $n=1000$. (a) Example of simulation trajectory (state on the ring versus time).
(b) Periodogram (averaged over 20 realizations), the green line is a fit with a Lorentzian function.}
\label{fig:simul_rw_star}
\end{center}
\end{figure}

\subsection{$1/f$ noise criterion}

When the eigenspectrum of $\mathbf{G}$ is highly non-degenerate, the power spectral density can present more complex shapes
beyond the Lorentzian spectrum.
In fact such a spectrum, resulting from the superposition of relaxation processes, has often been proposed 
as an explanation for $1/f$ noises \cite{Milotti2002}.
It often assumes an uniform distribution of the eigenfrequencies over a given range.
More generally, as already observed by \cite{Erland2011}, the power spectrum is determined by the complete eigenstructure
$(\omega_{k}, \gamma_{k})$ of the generator.
In particular, if $\omega_{k}$ and $\gamma_{k}$ obey appropriate power laws then the power spectral density exhibits a $1/f$ like noise,
more precisely the criterion is:

\begin{criterion}
Assume that the eigenvalues of $\mathbf{G}$ scales as $\omega_{k} \sim k^{\alpha}$ with $\alpha > 0$,
and that the coupling strength of the eigenvectors to the state vector $\mathbf{x}$ scales as $\gamma_{k} \sim k^{\beta}$.

If $\alpha  > |2 \beta + 1|$, then the power spectral density scales as $S_{X} (\omega) \sim \omega^{\zeta}$, where
\begin{equation}
\zeta = \frac{2 \beta - \alpha + 1}{\alpha}.
\end{equation}
It satisfies $-2 < \zeta < 0$.
An exact $1/f$ noise ($\zeta = -1$) is obtained for $\beta = -1/2$, for any $\alpha > 0$.
\label{prop:psd_scaling}
\end{criterion}

In practice the scaling law may only be valid over a finite range of eigenfrequencies and to achieve a $1/f^{\alpha}$ noise spectrum
the eigenspectrum of $\mathbf{G}$ must be highly degenerate with eigenvalues spanning over at least one frequency decade.
As we will see now, this property appears for queues (birth-death processes) in the heavy traffic regime.

\section{Application to the M/M/1 queue in the heavy traffic regime}

A M/M/1 queue is a birth-death process with constant rates ($\mu = \mu_{n}$, $\lambda = \lambda_{n}$ for all $n$).
The associated graph is presented in Fig.~\ref{fig:graph_mm1}.
It models a queue with a single server where arrival time and server processing time obey a exponential distribution
with time constants independent of the number of customers in the queue.
This is the "simplest" queue model and many closed-form results are known.

\begin{figure}[!h]
\begin{center}
\includegraphics[width=5cm,keepaspectratio]{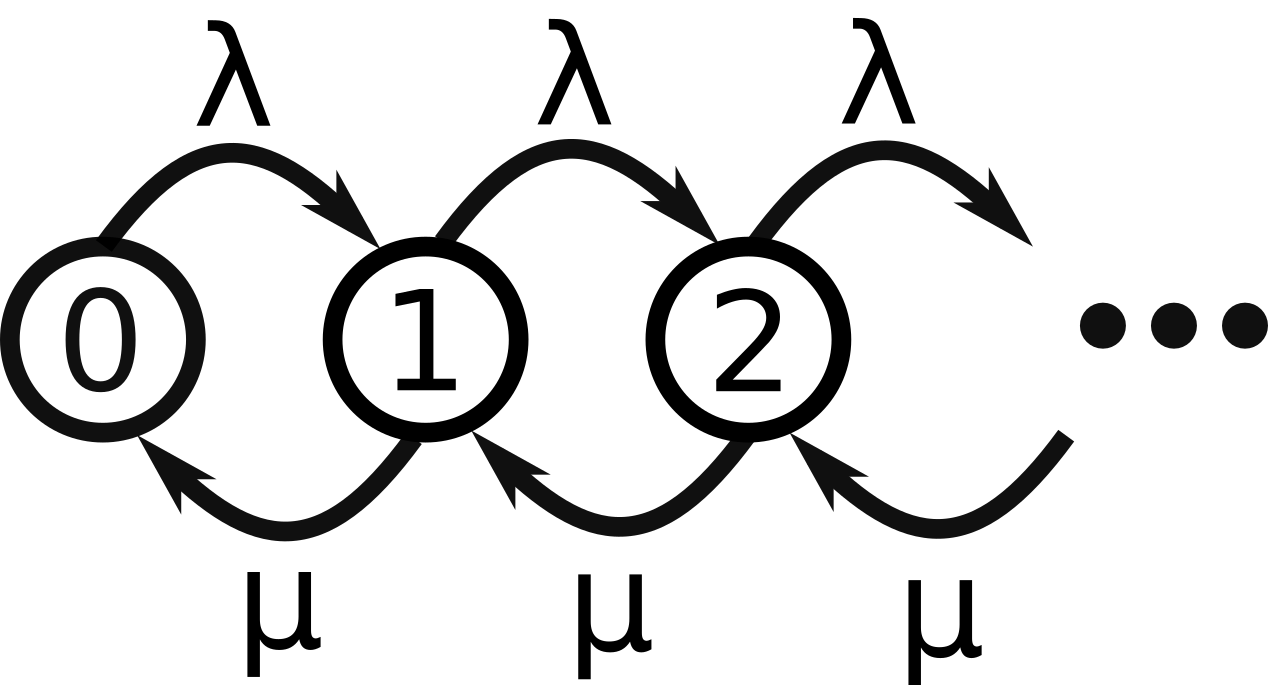}
\caption{Transition graph of the M/M/1 queue.}
\label{fig:graph_mm1}
\end{center}
\end{figure}

The infinitesimal generator of this process is a tridiagonal matrix:

\begin{equation}
\mathbf{G} =
\begin{pmatrix}
& \lambda  & -\lambda      &               &            &        \\
& -\mu     & \lambda + \mu & -\lambda      &            &        \\
&          & -\mu          & \lambda + \mu & -\lambda   &        \\
&          &               & \ddots        & \ddots     & \ddots \\
\end{pmatrix},
\label{eq:inf_gen_mm1_process}
\end{equation}

The server utilization is defined as $\rho \triangleq \lambda / \mu$.
The process is stable when $\rho < 1$, that is when the long term number of customers in the queue is finite.
The average number of customers in the queue is $m = \rho / (1 - \rho)$ and the variance is $\sigma^{2} = \rho / (1 - \rho)^{2}$.
The Markov chain has a stationary population which satisfies $\pi_{k} = \left( 1 - \rho \right) \rho^{k}$.

\subsection{Tridiagonal Toeplitz matrix}

The tridiagonal Toeplitz matrix is a $n \times n$ matrix with constant upper diagonal, lower diagonal and diagonal:
\begin{equation}
T_{n}(a, b, c) \triangleq
\begin{pmatrix}
& b & c &        &        &   \\
& a & b & c      &        &   \\
&   & a & \ddots & \ddots &   \\
&   &   & \ddots & b      & c \\
&   &   &        & a      & b \\
\end{pmatrix},
\end{equation}

In the limit $n \rightarrow \infty$, the generator Eq.~(\ref{eq:inf_gen_mm1_process}) can be approximated
by $T_{n}(a, b, c)$, where $a = -\mu$, $b = \mu + \lambda$, and $c = -\lambda$.
The only difference with the M/M/1 generator being the first row does not satisfies the condition $b + c = 0$.

The eigenvalues and eigenvectors of $T_{n}(a, b, c)$ have simple closed-form expressions.
The eigenvalues (indexed in increasing order) in the case $a, c < 0$ are (see for example \cite{Noschese2013}):
\begin{equation}
\omega_{k} = b - 2 \sqrt{ac} \cos \left( \frac{k \pi}{n+1} \right).
\label{eq:eigenvalue_toeplitz}
\end{equation}
The right eigenvectors are $\mathbf{v}_{k} = \left( v_{k, i} \right)_{1 \leq i \leq n}$ where
\begin{equation}
v_{k, i} = \left( \frac{a}{c} \right)^{i/2} \sin \left( \frac{k i \pi}{n+1} \right),
\end{equation}
and the left eigenvectors $\mathbf{w}_{k} = \left( w_{k, i} \right)_{1 \leq i \leq n}$ are:
\begin{equation}
w_{k, i} = \left( \frac{c}{a} \right)^{i/2} \sin \left( \frac{k i \pi}{n+1} \right).
\end{equation}

When the matrix is symmetric ($c = a$), the eigenvectors are the discrete Fourier transform basis,
this is because the infinitesimal generator is the Laplacian matrix of the graph associated to the chain.
In fact in App.~\ref{app:graph_fourier_transform} we explicit the relation between the power spectral density and the graph Fourier transform.
In the case of a birth-death process, the associated graph is a discrete semi-open unidimensional string
whose eigenmodes (the eigenvectors of the graph Laplacian matrix) are the discrete Fourier transform basis.

\subsection{Light traffic regime}

In the light traffic regime, the server utilization is low $\rho \ll 1$, and there are few customers in the queue.
Choosing $\mu = 1$ and $\lambda = \varepsilon$, the server utilization is $\rho = \varepsilon$.
The average number of customers is $m \sim \varepsilon$ and the variance is $\sigma^{2} \sim \varepsilon$,
in other words $\sigma \sim \sqrt{m}$, and the underlying distribution behaves as a Poisson distribution.
In the limit $\varepsilon \rightarrow 0$, the eigenvalues satisfy 
\begin{equation}
\omega_{k} = 1 + \varepsilon - 2 \sqrt{\varepsilon} \cos \left( \frac{k \pi}{n+1} \right) \underset{\varepsilon \rightarrow 0}{\longrightarrow} 1,
\end{equation}
so that all the eigenfrequencies are degenerate and the power spectral density is Lorentzian.

\subsection{Scaling laws in the heavy traffic regime}

In the heavy traffic regime the server utilization is high ($\rho \sim 1$) and the number of customers in the queue is large.
This regime is known to converge to a reflected Brownian motion \cite{Kingman1962, Iglehart1970},
that is a Brownian motion over the half-plane $\mathbb{R}^{+}$.
Here, we set $\lambda = 1$ and $\mu = \lambda + \varepsilon = 1 + \varepsilon$ where $\varepsilon \ll 1$.
Therefore the utilization is $\rho \sim 1 - \varepsilon$, the average number of customers is $m \sim 1/\varepsilon$
and the variance is $\sigma^{2} \sim 1 / \varepsilon^{2}$.
So in this regime the average number of customers equals the standard deviation ($\sigma \simeq m$),
showing that the underlying distribution has much larger fluctuations than a Poisson process.

From Eq.~(\ref{eq:eigenvalue_toeplitz}), the eigenvalues of the generator satisfies:
\begin{equation}
\omega_{k} = 2 + \varepsilon - 2 \sqrt{1 + \varepsilon} \cos \left( \frac{k \pi}{n+1} \right)
 \underset{\varepsilon \rightarrow 0}{\sim} 2 \left[ 1 - \cos \left( \frac{k \pi}{n} \right) \right]
 \underset{k \ll n}{\sim} \left( \frac{k \pi}{n} \right)^{2}.
\label{eq:eigenfrequency_mm1}
\end{equation}
In other words the eigenfrequencies follow the power law $\omega_{k} \sim k^{2}$ ($\alpha = 2$).

In Fig.~\ref{fig:scaling_tridiagonal_toeplitz} (a), we compare this approximation with the exact analytical result for the Toeplitz matrix,
and with the numerical eigenvalues for the M/M/1 generator.
We observe that for large $n$, the eigenfrequencies are highly non-degenerate and are densely covering many order of magnitudes.
In appendix \ref{app:birth_deat_heavy_traffic}, we see that this behavior is generalized among birth-death processes because
the characteristic polynomial of a tridiagonal matrix is an orthogonal polynomial whose roots are densely spanning an interval.

\begin{figure}[!h]
\begin{center}
\includegraphics[width=12cm,keepaspectratio]{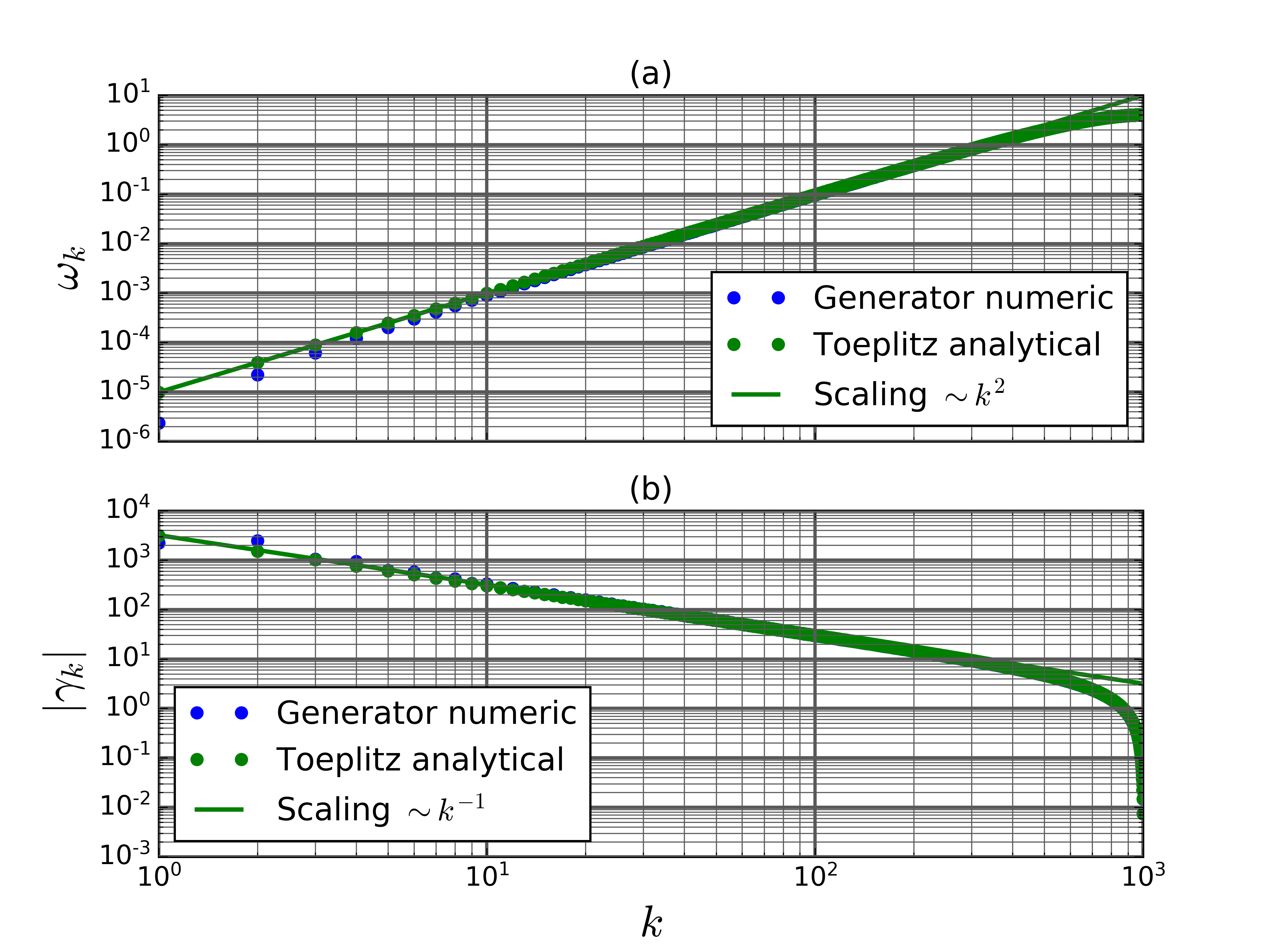}
\caption{Power law scalings for the M/M/1 infinitesimal generator. 
Evolution of (a) the eigenvalues $\omega_{k}$ and (b) the coefficients $|\gamma_{k}|$ for the generator and the
tridiagonal Toeplitz matrix versus $k$ for $n=10^{3}$, $\varepsilon = 10^{-4}$.
The blue dots are the result of the numerical eigendecomposition of the M/M/1 queue generator,
the green dots are the analytical result for the Toeplitz matrix
 and the solid line are the power law approximations $\omega_{k} \sim (k \pi / n)^{2}$
and $|\gamma_{k}| \sim \sqrt{\varepsilon} n^{2} / (\pi k)$.}.
\label{fig:scaling_tridiagonal_toeplitz}
\end{center}
\end{figure}

Since we analyze the number of customers in the queue versus time, the weight associated to state $i$ is simply $x_{i} = i$.
The coupling coefficient $\gamma_{k}$ satisfies (using the spectral projector $\boldsymbol\Pi_{k} = \mathbf{w}^{\top}_{k} \mathbf{v}_{k}$):
\begin{align}
\gamma_{k}^{2} &= \left\langle \mathbf{x}, \mathbf{w}^{\top}_{k} \mathbf{v}_{k} \mathbf{x} \right\rangle _{\boldsymbol\pi} \\
&= \left( 1 - \rho \right) \sum_{i,j} ij \rho^{(i + j)/2} \sin \left( \frac{k i \pi}{n+1} \right) \sin \left( \frac{k j \pi}{n+1} \right) \\
&= \left( 1 - \rho \right) \left[ \sum_{i} i \rho^{i/2} \sin \left( \frac{k i \pi}{n+1} \right) \right]^{2}.
\end{align}

In the limit $\varepsilon \rightarrow 0$, $1 - \rho = \varepsilon + o(\varepsilon^{2})$ and $\rho^{i/2} = 1 - i \varepsilon / 2 + o(\varepsilon^{2})$,
hence
\begin{equation}
\gamma_{k}^{2} = \varepsilon \; \left[ \sum_{i=1}^{n} i \sin \left( \frac{k i \pi}{n+1} \right) \right]^{2} + o(\varepsilon^{2}).
\end{equation}

We can evaluate the sum:
\begin{equation}
\sum_{i=1}^{n} i \sin \left( \frac{k i \pi}{n+1} \right) = - \frac{(n + 1)}{2} \frac{\cos \left( \frac{\pi k (2n+1)}{2n+2} \right)}{\sin \left( \frac{\pi k}{2n+2} \right)}
\underset{n \rightarrow \infty}{\sim} - \frac{n \cos \left( \pi k \right)}{2 \sin \left( \frac{\pi k}{2n} \right)}
\underset{k \ll n}{\sim} (-1)^{k + 1} \frac{n^{2}}{\pi k}.
\end{equation}
Therefore $\left| \gamma_{k} \right| \sim \sqrt{\varepsilon} n^{2} / (\pi k)$ and $\gamma_{k}$ scales as $k^{-1}$ ($\beta = -1$).
In Fig.~\ref{fig:scaling_tridiagonal_toeplitz} (b), we compare this approximation with the exact analytical result.

The scaling exponents $(\alpha, \beta) = (2, -1)$ satisfy $\alpha = 2 > |2\beta + 1| = 1$, and from Criterion~\ref{prop:psd_scaling} the 
power spectral density scales as $S_{X} (\omega) \sim \omega^{\zeta}$, where $\zeta = (2 \beta - \alpha + 1) / \alpha = -3/2$.
In other words, the power spectral density of the number of customers in a M/M/1 queue in the heavy traffic regime exhibits a $1/f^{3/2}$ noise.

\subsection{Heavy traffic simulations}

The M/M/1 queue is simulated using the Gillepsie algorithm \cite{Gillespie1976}, where each iteration compute the remaining time in a given state
and the jump to the next step.
Since the remaining time in a given state is random, the Gillepsie algorithm does not produce uniformly sampled data,
so the result is interpolated to obtain the uniform sampling required by the Fast Fourier Transform (FFT) algorithm.
The power spectral density is estimated from the periodogram, obtained by computing the modulus square of the FFT of the data.
Periodograms of many simulation trajectories are averaged to improve the power spectral density estimation.

\begin{figure}[!h]
\begin{center}
\includegraphics[width=15cm,keepaspectratio]{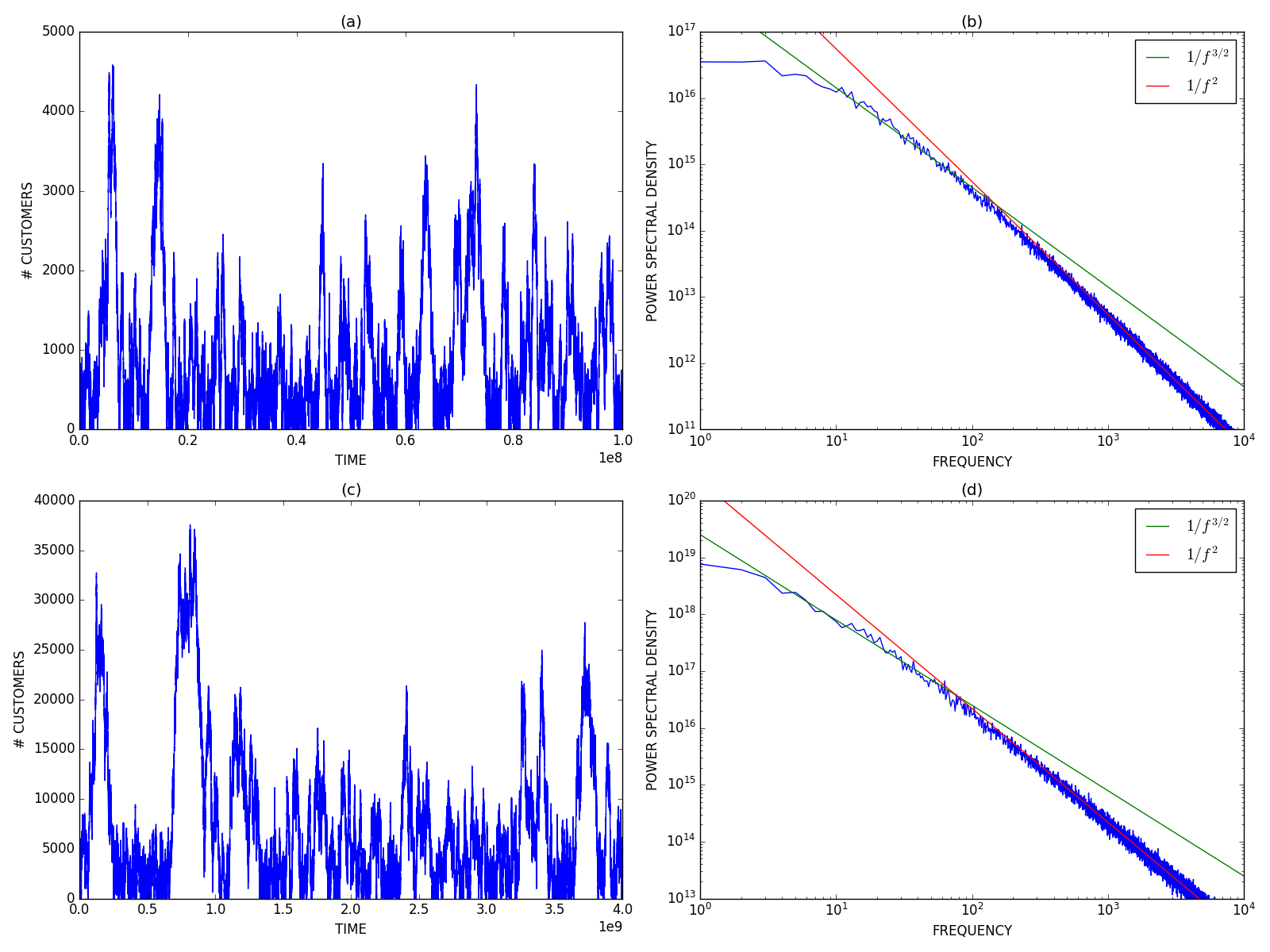}
\caption{Simulation of the M/M/1 queue in the heavy traffic regime. Example of simulation trajectories (number of customers in the queue versus time)
for (a) $\varepsilon = 10^{-3}$ and (c) $\varepsilon = 10^{-4}$. Periodograms for (b) $\varepsilon = 10^{-3}$ (averaged over 69 realizations)
and (d) $\varepsilon = 10^{-4}$ (averaged over 34 realizations).}
\label{fig:simul_mm1}
\end{center}
\end{figure}

Simulation results are presented Fig.~\ref{fig:simul_mm1} for $\varepsilon = 10^{-3}$ and $\varepsilon = 10^{-4}$.
We check that the mean value satisfies $m \sim 1 / \varepsilon$: for $\varepsilon = 10^{-3}$, $m = 987$ (expected $10^{3}$),
and for $\varepsilon = 10^{-4}$, $m = 10357$ (expected $10^{4}$).
We observe the periodogram scales as $1/f^{3/2}$ over at least a frequency decade.
Note also the $1/f^{2}$ scaling at higher frequencies, showing the spectrum is consistent with a sum of Lorentzians.

\section{Continuous time random walk on a ring}

To further understand the above result, that is a continuous time random walk on a semi-open string,
we fold the string onto itself to form a ring and consider the resulting continuous time random walk on a ring.
This process could for example model the phase diffusion of an oscillator.

Interestingly, a $1/f^{3/2}$ noise has also been observed for a random walk on a ring \cite{Erland2007, Yadav2021}.
Here we confirm this result and show that the PSD of a continuous time random walk on a ring obeys a $1/f^{3/2}$ scaling.
Indeed, we observe that the eigenstructure of the circular random walk generator is similar with the one of the M/M/1 queue generator,
more specifically the scaling law coefficients $(\alpha, \beta)$ are the same.

\begin{figure}[!h]
\begin{center}
\includegraphics[width=5cm,keepaspectratio]{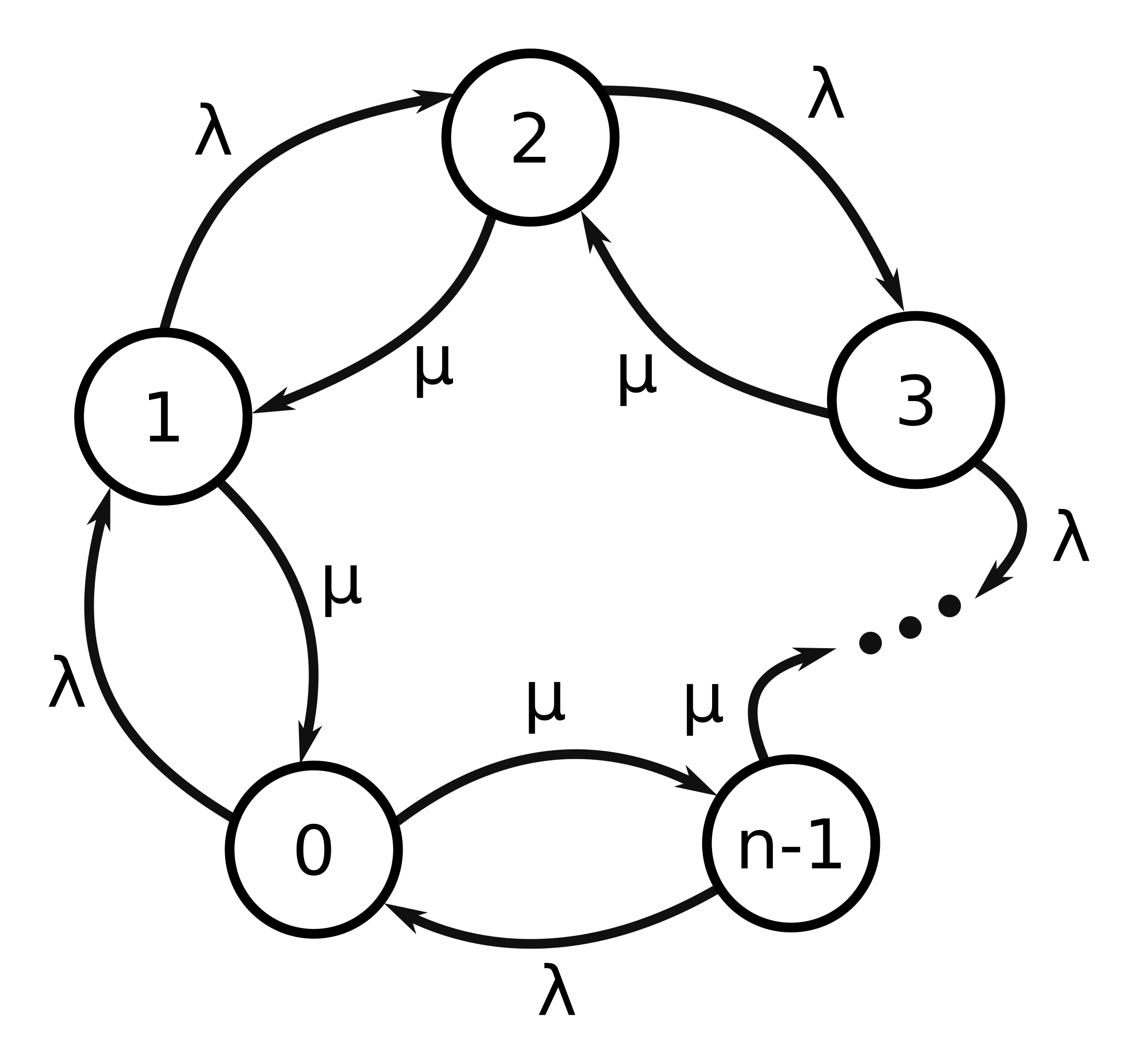}
\caption{Transition graph of the continuous time random walk on a ring.}
\label{fig:graph_rw_ring}
\end{center}
\end{figure}

\subsection{Generator, eigenvalues and eigenvectors}

We consider a ring of $n$ states with constant transition rates clockwise $\lambda$ and counterclockwise $\mu$ 
(transition graph is presented in Fig.~\ref{fig:graph_rw_ring}).
The infinitesimal generator of this process is a $n \times n$ matrix:

\begin{equation}
\mathbf{G} =
\begin{pmatrix}
& \lambda + \mu  & -\lambda      &               &            & -\mu          \\
& -\mu           & \lambda + \mu & -\lambda      &            &               \\
&                & -\mu          & \lambda + \mu & \ddots     &               \\
&                &               & \ddots        & \ddots     & -\lambda      \\
& -\lambda       &               &               & -\mu       & \lambda + \mu \\
\end{pmatrix}.
\label{eq:inf_gen_rw_ring_process}
\end{equation}

It is a right circulant matrix, which is strongly related to the discrete Fourier transform.
In what follows, we define $w_{n} \triangleq \exp(i 2 \pi / n)$ so that $w_{n}^{k}$ is the $k$th root of unity.
The eigenvalues of the generator are:
\begin{equation}
\omega_{k} = \lambda + \mu - \lambda w_{n}^{k} - \mu w_{n}^{(n-1) k}.
\label{eq:eigenvalues_gen_rw_ring_process}
\end{equation}
We further assume that clockwize and counterclockwise circulations are identical, that is $\lambda = \mu \triangleq 1$.
In that case the generator is the Laplacian matrix of an undirected cycle graph $C_{n}$.
As a consequence the generator is a symmetric matrix and the Markov process is reversible.
The eigenvalues satisfy:
\begin{equation}
\omega_{k} = 2 - 2 \cos \left( \frac{2 \pi k}{n} \right) \underset{k \ll n}{\sim} 4 \left( \frac{\pi k}{n} \right)^{2}.
\end{equation}
The eigenvalues scaling exponent is therefore $\alpha = 2$.
This result is very similar with the M/M/1 case Eq.~(\ref{eq:eigenfrequency_mm1}), up to a factor 2 in the cosine argument.

The eigenvectors $\mathbf{v}_{k} = \left( v_{k, q} \right)_{1 \leq q \leq n}$ are the same for any circulant matrix,
in particular it does not depend on $\lambda$ and $\mu$:
\begin{equation}
v_{k, q} = \frac{1}{\sqrt{n}} w_{n}^{(q-1)k}.
\label{eq:eigenvector_gen_rw_ring_process}
\end{equation}
Because the generator is symmetric, the right and left eigenvectors are identical.
The spectral projector is $\boldsymbol\Pi_{k} = \mathbf{v}^{\dagger}_{k} \mathbf{v}_{k}$ where $\mathbf{v}^{\dagger}_{k}$ is the
hermitian conjugate of $\mathbf{v}_{k}$, it satisfies:

\begin{equation}
\left( \boldsymbol\Pi_{k} \right)_{q,r} = \frac{1}{n} w_{n}^{(1-q)k} w_{n}^{(r - 1)k} = \frac{1}{n} w_{n}^{(r-q)k}.
\end{equation}

\subsection{$1/f^{3/2}$ noise}

The coupling coefficient $\gamma_{k}^{2}$ for $x_{q} = q$ is:

\begin{align}
\gamma_{k}^{2} &= \sum_{q,r} \pi_{q} q r  \left( \boldsymbol\Pi_{k} \right)_{q,r} \\
&= \frac{1}{n^{2}} \sum_{q,r} q r w_{n}^{(r-q)k} \\
&= \frac{1}{n^{2}} \left( \sum_{q=0}^{n-1} q w_{n}^{-qk} \right) \left( \sum_{r=0}^{n-1} r w_{n}^{rk} \right) \\
&= \frac{1}{n^{2}} \; \frac{n}{2 i e^{i \pi k / n} \sin (\pi k / n)} \; \frac{n e^{i \pi k / n}}{-2 i \sin (\pi k / n)} \\
&= \frac{1}{4 \sin^{2} (\pi k / n)}.
\end{align}

Here we used the fact that the stationary distribution for the random walk on a ring is uniform \cite{2947045},
that is $\pi_{q} = 1/n$ for all $q$. Finally,
\begin{equation}
\gamma_{k} = \frac{1}{2 \sin (\pi k / n)} \underset{k \ll n}{\sim} \frac{n}{2 \pi k}.
\end{equation}
In other word, the coupling coefficient scaling exponent is $\beta = -1$.

We observe that the continuous time random walk on a ring has the same scaling coefficients $(\alpha, \beta) = (2, -1)$ 
that the M/M/1 queue in the heavy traffic limit, and the PSD also scales as $1/f^{3/2}$.
Intuitively this is because the eigenmodes of the semi-open string of the M/M/1 queue in the heavy traffic limit are circulating far away from
the origin, just like the circulating counter-propagating eigenmodes on a ring (similarly to the circulating modes in a circular optical resonator).
We confirmed this result by a numerical simulation Fig.~\ref{fig:simul_rw_ring}.

\begin{figure}[!h]
\begin{center}
\includegraphics[width=15cm,keepaspectratio]{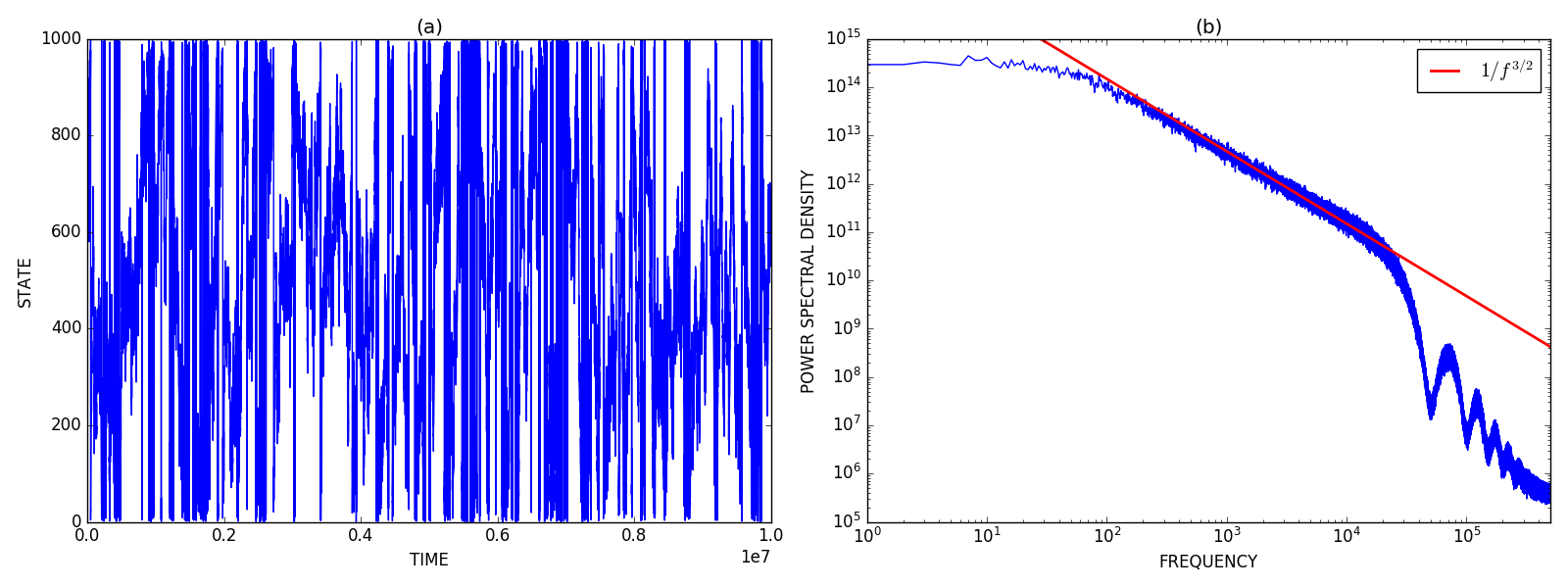}
\caption{Simulation of the continuous time random walk on a ring of size $n=1000$. (a) Example of simulation trajectory (state on the ring versus time).
(b) Periodogram (averaged over 30 realizations), the red line is the $1/f^{3/2}$ scaling.}
\label{fig:simul_rw_ring}
\end{center}
\end{figure}


\subsection{Sequence of pulses: $1/f^{1/2}$ noise}

Existence of $1/f$ noise in pulse sequences has been demonstrated in \cite{Kononovicius2023}, for example.
Here we show that such signal can be obtained within the framework develop in this article.
Until now we have only consider the case where the weight equals the state index ($x_{i} = i$), but other options are possible.
By choosing a Kronecker delta for the graph weights: $x_{q} = \delta_{q}$,
we obtain a pulsed signal where a pulse is emitted when the stochastic trajectory crosses the $n=0$ state.
Note that, because the number of internal states $n$ is larger than 2, this process is not a telegraph noise and,
as we will see, the power spectral density is not Lorentzian.

The coupling coefficient simply satisfies:
\begin{equation}
\gamma_{k}^{2} = \frac{1}{n^{2}} \sum_{q,r} \delta_{q} \delta_{r} w_{n}^{(r-q)k} = \frac{w_{n}^{0}}{n^{2}} = \frac{1}{n^{2}}.
\end{equation}
Hence, there is no scaling of $\gamma_{k}$ with $k$, that is $\beta = 0$.
This process has scaling exponents $(\alpha, \beta) = (2, 0)$,
which satisfies $\alpha = 2 > |2 \beta + 1| = 1$ and, by Criterion~\ref{prop:psd_scaling},
the power spectral density scales as $S(\omega) \sim \omega^{-1/2}$,
which is what we observe in the numerical simulation Fig.~\ref{fig:simul_rw_ring_delta}.

\begin{figure}[!h]
\begin{center}
\includegraphics[width=15cm,keepaspectratio]{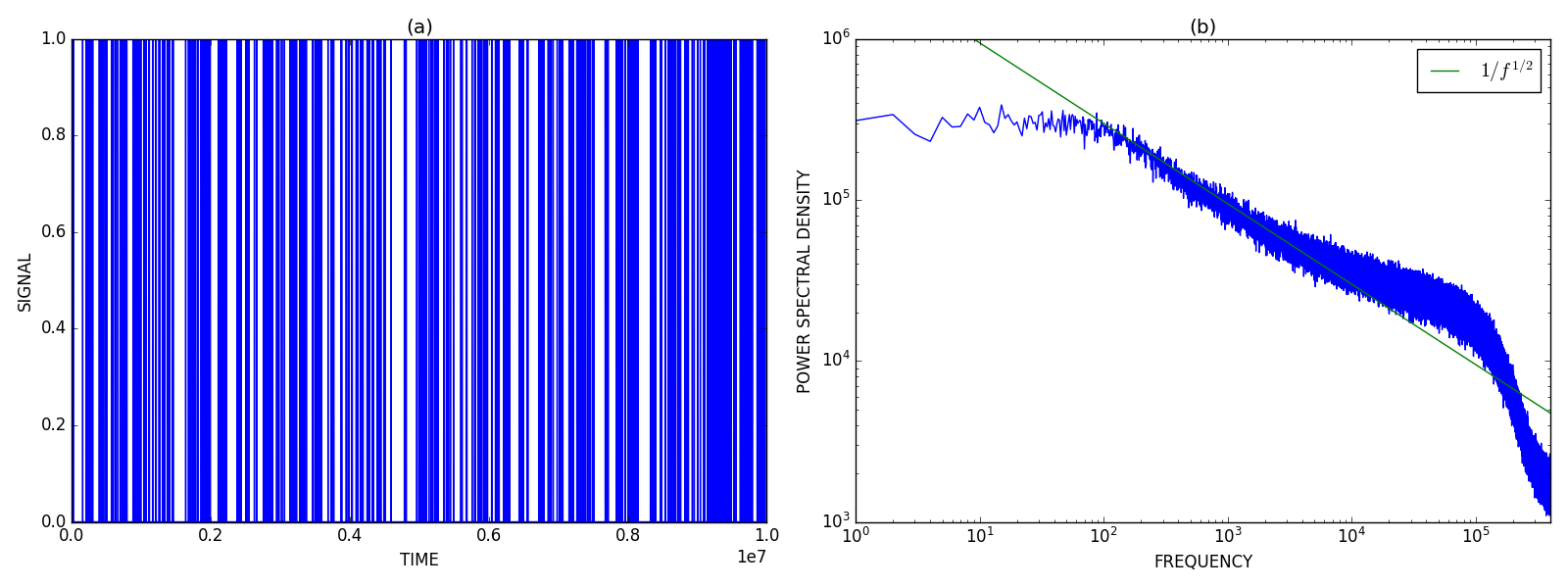}
\caption{Simulation of a sequence of pulses with underlying random walk on a ring of size $n=1000$.
(a) Example of simulation trajectory.
(b) Periodogram (averaged over 90 realizations), the green line is the $1/f^{1/2}$ scaling.}
\label{fig:simul_rw_ring_delta}
\end{center}
\end{figure}


\section{Conclusion}
Using a simple criterion relating the $1/f$ noise exponent to scalings of the eigenstructure of the infinitesimal generator,
we observed that the M/M/1 queue in the heavy traffic and the random walk on a ring exhibit a $1/f^{3/2}$ noise.
The existence of such a non-trivial behavior within the simplest model in queuing theory is an 
interesting example of the emergence of complex phenomena from simple rules.
We also give an example of a sequence of pulses resulting from an underlying random walk on a ring and resulting in a $1/f^{1/2}$ noise.

The M/M/1 queue in the heavy traffic is a reflected Brownian motion \cite{enwiki:1069489542}, 
and interestingly $1/f$ noise has also been observed in other stochastic processes with limited or modified boundaries.
We observed it for the random walk on a ring, that is a process with circular boundaries, but it has also been demonstrated in
non-linear stochastic equations with reflective boundaries \cite{Ruseckas2014}.
The reflective boundaries define the eigenmodes of the Markov chain graph (similarly to the eigenmodes of a resonator) 
and therefore the eigenstructure of the generator.
The relationship between reflections in a Markov chain and $1/f$ noise is an interesting problem to explore.

As shown in App.~{\ref{app:graph_fourier_transform}}, the weight vector $\mathbf{x} = (x_{i})$ can be understood as a signal on the transition graph,
and the power spectral spectral density is related to the graph Fourier transform of this signal.
This shows that tools from the field of graph signal processing \cite{Ricaud2019} can be used to study $1/f$ noise resulting from graph random walks.
Conversely, the analysis of power spectral densities could be used to study graph signals.
For example simulating a random walk on a large graph might be less computationally intensive than performing the full eigendecomposition
and still provide useful insights by looking at the power spectral density and / or the autocorrelation function
(power law scalings could help in classifying graphs).


\appendix

\section{Proof of Prop.~\ref{prop:autocorrelation}}\label{app:autocorrelation}

\begin{proof}
Because the process is stationary $\left\langle X_{t+\tau} \right\rangle = \left\langle X_{t} \right\rangle$, and the autocorrelation function satisfies:
\begin{equation}
\mathcal{C}_{X} (\tau) = \lim_{t \rightarrow \infty} \left[ \left\langle X_{t+\tau} X_{t} \right\rangle - \left\langle X_{t} \right\rangle^{2} \right].
\end{equation}
Using the Markovian property:
\begin{equation}
P(X_{t+\tau} = j \wedge X_{t} = i) = P(X_{t+\tau} = j | X_{t} = i) P(X_{t} = i),
\end{equation}
and because the CTMC is homogeneous
\begin{equation}
P(X_{t+\tau} = j | X_{t} = i) = P(X_{\tau} = j | X_{0} = i) = \mathbf{P}_{ij} (\tau), 
\end{equation}
the first term is
\begin{equation}
\left\langle X_{t+\tau} X_{t} \right\rangle = \sum_{i,j} x_{i} x_{j} P(X_{t+\tau} = j \wedge X_{t} = i)
= \sum_{i,j} x_{i} x_{j} \mathbf{P}_{ij} (\tau) P(X_{t} = i).
\end{equation}

Moreover the chain is irreducible with a finite state space so it as an unique stationary distribution $\boldsymbol\pi = (\pi_{i})_{i}$
which satisfies $\lim_{t \rightarrow \infty} P(X_{t} = i) = \pi_{i}$, so that
\begin{equation}
\lim_{t \rightarrow \infty}\left\langle X_{t+\tau} X_{t} \right\rangle = \sum_{i,j} \pi_{i} x_{i} x_{j} \mathbf{P}_{ij} (\tau)
= \left\langle \mathbf{x}, \mathbf{P} (\tau) \mathbf{x} \right\rangle _{\boldsymbol\pi}.
\end{equation}
The stationary average value of the process is
\begin{equation}
\lim_{t \rightarrow \infty} \left\langle X_{t} \right\rangle = \lim_{t \rightarrow \infty} \sum_{i} x_{i} P(X_{t} = i) = \sum_{i} x_{i} \pi_{i}
                                                                = \left\langle \mathbf{x}, \mathbf{1} \right\rangle _{\boldsymbol\pi}.
\end{equation}
Since $\left\langle \mathbf{x}, \mathbf{1} \right\rangle^{2} = \sum_{i,j} x_{i} x_{j} \pi_{i} \pi_{j}$, it results
\begin{equation}
\mathcal{C}_{X} (\tau) = \sum_{i,j} x_{i} x_{j} \pi_{i} \left[ \mathbf{P}_{ij} (\tau) - \pi_{j} \right].
\end{equation}
Moreover, $X$ is irreducible and has a stationary distribution, so the theorem of convergence to invariant distribution for a 
continuous time Markov chain is satisfied,
and $\mathbf{P}_{ij} (\infty) \triangleq \lim_{\tau \rightarrow \infty} \mathbf{P}_{ij} (\tau) = \pi_{j}$. Therefore,
\begin{equation}
\mathcal{C}_{X} (\tau) = \sum_{i,j} x_{i} x_{j} \pi_{i} \left[ \mathbf{P}_{ij} (\tau) - \mathbf{P}_{ij} (\infty) \right]
= \left\langle \mathbf{x}, \left[ \mathbf{P} (\tau) - \mathbf{P} (\infty) \right] \mathbf{x} \right\rangle _{\boldsymbol\pi},
\end{equation}
which is the expected result.
\end{proof}

\section{Proof of Prop.~\ref{prop:gen_fund_mat}}\label{app:proof_gen_fund_mat}

Here we explicit the relationship between the fundamental matrix and the infinitesimal generator.
In this appendix $\sigma (\mathbf{A})$ designates the spectrum of the matrix $\mathbf{A}$, and $\mathbf{A} \oplus \mathbf{B}$
is the direct sum of $\mathbf{A}$ and $\mathbf{B}$. Moreover, $\mathds{1}_{1}$ is the $1 \times 1$ identity matrix and
$\mathbf{0}_{1}$ is the $1 \times 1$ zero matrix. 

The result is straightforward to show if the generator is diagonalizable, as this is the case for a reversible process,
however the result is valid for any generator of an irreducible chain.
We start with the following lemma for the decomposition of the transition matrix,
mostly resulting from the Perron-Frobenius theorem for transition matrices:

\begin{lemma}[\cite{campbell1979}, Lemma 8.2.2]
If $\mathbf{P}$ is the transition matrix of an ergodic (irreducible) Markov chain then it exists an invertible matrix $\mathbf{S}$ so that
$\mathbf{P} = \mathbf{S} \left[ \mathds{1}_{1} \oplus \mathbf{K} \right] \mathbf{S}^{-1}$,
with $1 \notin \sigma (\mathbf{K})$. Moreover, if the chain is regular then $\lim_{n \rightarrow \infty} \mathbf{K}^{n} = \mathbf{0}$.
\label{lem:fund_trans_mat}
\end{lemma}

As a corollary, we have an equivalent result for the core-nilpotent decomposition of the infinitesimal generator:
\begin{lemma}
If $\mathbf{G}$ is the generator of an irreducible HCTMC over a finite state space, then it exists an invertible matrix $\mathbf{S}$ so that
$\mathbf{G} = \mathbf{S} \left[ \mathbf{0}_{1} \oplus \boldsymbol{\Gamma} \right] \mathbf{S}^{-1}$
where:
\begin{enumerate}[(i)]
\item $\boldsymbol{\Gamma}$ is invertible.
\item $\lim_{\tau \rightarrow \infty} e^{-\boldsymbol{\Gamma} \tau} = \mathbf{0}$.
\item for all $\bar{\omega} \in \sigma(\boldsymbol\Gamma)$, $\mathrm{Re} \left( \bar{\omega} \right) > 0$.
\end{enumerate}
\label{lem:fund_prop_gene}
\end{lemma}

\begin{proof}
Using Eq.~(\ref{eq:sol_gen}) with $\tau = 1$, one has $\mathbf{P}(1) = e^{-\mathbf{G}}$.
Since $\mathbf{P}(1)$ is a transition matrix of an irreducible chain, by Lemma~\ref{lem:fund_trans_mat}
it exists an invertible matrix $\mathbf{S}$ so that
$\mathbf{P}(1) = \mathbf{S} \left[ \mathds{1}_{1} \oplus \mathbf{K} \right] \mathbf{S}^{-1}$, with $1 \notin \sigma (\mathbf{K})$.

Noting that $\mathbf{P}(1)$ is invertible (the inverse matrix being $e^{+\mathbf{G}}$), it follows that $\mathbf{K}$ is also invertible
because $\det \mathbf{K} = \det \mathbf{P}(1) \neq 0$.
Since the logarithm of an invertible matrix exists \cite[Th. 1.27]{Higham2008}, but is not uniquely defined,
let $-\boldsymbol{\Gamma}$ be a logarithm of $\mathbf{K}$ so that $\mathbf{K} = \exp(-\boldsymbol{\Gamma})$, and
\begin{equation}
\mathbf{P}(1) = \mathbf{S} \left[ e^{-\mathbf{0}_{1}} \oplus e^{-\boldsymbol{\Gamma}} \right] \mathbf{S}^{-1}
              = \exp \left(- \mathbf{S} \left[ \mathbf{0}_{1} \oplus \boldsymbol{\Gamma} \right] \mathbf{S}^{-1} \right),
\end{equation}
and therefore $\mathbf{G} = \mathbf{S} \left[ \mathbf{0}_{1} \oplus \boldsymbol{\Gamma} \right] \mathbf{S}^{-1}$.
Moreover, since $1 \notin \sigma (\mathbf{K})$ then $0 \notin \sigma (\boldsymbol{\Gamma})$ and $\boldsymbol{\Gamma}$ is invertible,
proving (i).

Given that a continuous time Markov chain over a finite state space is regular \cite{4025628}, Lemma~\ref{lem:fund_trans_mat} yields
$\lim_{n \rightarrow \infty} \mathbf{K}^{n} = 0$ and (ii) follows
\begin{equation}
\lim_{\tau \rightarrow \infty} e^{-\boldsymbol{\Gamma} \tau} = \lim_{n \rightarrow \infty} e^{-\boldsymbol{\Gamma} n}
                                                             = \lim_{n \rightarrow \infty} \left( e^{-\boldsymbol{\Gamma}} \right)^{n}
                                                             = \lim_{n \rightarrow \infty} \mathbf{K}^{n} = \mathbf{0}.
\end{equation}

Finally, to show (iii) we observe than if $\bar{\omega} \in \sigma(\boldsymbol\Gamma)$
then $e^{-\bar{\omega}} \in \sigma(\mathbf{K})$ and $| e^{-\bar{\omega}} | < 1$,
because $e^{-\bar{\omega}}$ is not the Perron-Frobenius eigenvalue of the stochastic matrix $\mathbf{P}(1)$,
therefore $\mathrm{Re} \left( \bar{\omega} \right) > 0$.
\end{proof}

We now use this result to prove Prop.~\ref{prop:gen_fund_mat}:

\begin{proof}
In this proof, we make use of the group inverse $\mathbf{A}^{\sharp}$ of $\mathbf{A}$.
It is a special case of the Drazin inverse for the matrix index of 1,
which applies to us as we are dealing with matrices whose zero eigenvalue as multiplicity 1 (see \cite{campbell1979} Theorem 8.2.1 and its proof).
In other words, if $\mathbf{A}$ admits the decomposition $\mathbf{A} = \mathbf{S} \left[\mathbf{0}_{1} \oplus \mathbf{Q} \right] \mathbf{S}^{-1}$
with $\mathbf{Q}$ invertible, then $\mathbf{A}^{\sharp} = \mathbf{S} \left[\mathbf{0}_{1} \oplus \mathbf{Q}^{-1} \right] \mathbf{S}^{-1}$.

By Lemma~\ref{lem:fund_prop_gene}, the transition matrix admits the following decomposition:
\begin{equation}
\mathbf{P} (\tau) = e^{- \mathbf{S} \left[ \mathbf{0}_{1} \oplus \boldsymbol{\Gamma} \right] \tau}
                  = \mathbf{S} \left[ \mathds{1}_{1} \oplus e^{-\boldsymbol\Gamma \tau} \right] \mathbf{S}^{-1},
\end{equation}
moreover using Lemma~\ref{lem:fund_prop_gene} (ii),
\begin{equation}
\mathbf{P} (\infty) = \lim_{\tau \rightarrow \infty} \mathbf{S} \left[ \mathds{1}_{1} \oplus e^{-\boldsymbol\Gamma \tau} \right] \mathbf{S}^{-1} 
                    = \mathbf{S} \left[ \mathds{1}_{1} \oplus \mathbf{0} \right] \mathbf{S}^{-1},
\end{equation}
so that $\mathbf{P} (\tau) - \mathbf{P} (\infty) = \mathbf{S} \left[ \mathbf{0}_{1} \oplus e^{-\boldsymbol\Gamma \tau} \right] \mathbf{S}^{-1}$.
Therefore, the fundamental matrix is decomposed as
\begin{equation}
\mathbf{Z} (\omega) = \mathbf{S} \left[ \mathbf{0}_{1} \oplus \mathbf{J}(\omega) \right] \mathbf{S}^{-1},
\end{equation}
where
\begin{equation}
\mathbf{J}(\omega) \triangleq \int_{0}^{\infty} e^{- \boldsymbol\Gamma \tau} \cos (\omega \tau) d\tau
= \frac{1}{2} \left[ \int_{0}^{\infty} e^{- \boldsymbol\Gamma \tau} e^{+i \omega \tau} d\tau + \int_{0}^{\infty} e^{- \boldsymbol\Gamma \tau} e^{-i \omega \tau} d\tau \right].
\end{equation}
Recall that the Laplace transform of the exponential of a matrix $\mathbf{X}$ is 
$\int_{0}^{\infty} e^{-ts} e^{t\mathbf{X}} dt = (s\mathds{1} - \mathbf{X})^{-1}$, which is defined for all
$s \in \mathbb{C} \setminus \sigma(\mathbf{X})$, and in particular
\begin{equation}
\int_{0}^{\infty} e^{- \boldsymbol\Gamma \tau} e^{\pm i \omega \tau} d\tau = (\boldsymbol\Gamma \mp i \omega \mathds{1})^{-1},
\end{equation}
which is define for all $\omega \in \mathbb{R}$ since by Lemma~\ref{lem:fund_prop_gene} (iii) $\sigma(\boldsymbol\Gamma) \not\subset i\mathds{R}$.
As a consequence, because $\boldsymbol{\Gamma}$ is invertible (Lemma~\ref{lem:fund_prop_gene} (i)),
\begin{equation}
\mathbf{J}(\omega) = \frac{1}{2} \left[(\boldsymbol\Gamma -i \omega \mathds{1})^{-1} + (\boldsymbol\Gamma +i \omega \mathds{1})^{-1} \right]
= \boldsymbol\Gamma \left(\boldsymbol\Gamma^{2} + \omega^{2} \mathds{1} \right)^{-1}
= \left(\boldsymbol\Gamma + \omega^{2} \boldsymbol\Gamma^{-1} \right)^{-1}.
\end{equation}
Hence,
\begin{equation}
\mathbf{Z} (\omega) = \mathbf{S} \left[ \mathbf{0}_{1} \oplus \left(\boldsymbol\Gamma + \omega^{2} \boldsymbol\Gamma^{-1} \right)^{-1} \right] \mathbf{S}^{-1} \triangleq \mathbf{M}^{\sharp},
\end{equation}
where $\mathbf{M}^{\sharp}$ is the group inverse of
\begin{equation}
\mathbf{M} = \mathbf{S} \left[ \mathbf{0}_{1} \oplus \left(\boldsymbol\Gamma + \omega^{2} \boldsymbol\Gamma^{-1} \right) \right] \mathbf{S}^{-1}
= \mathbf{S} \left[ \mathbf{0}_{1} \oplus \boldsymbol\Gamma \right] \mathbf{S}^{-1} + \omega^{2} \mathbf{S} \left[ \mathbf{0}_{1} \oplus \boldsymbol\Gamma^{-1} \right] \mathbf{S}^{-1}
= \mathbf{G} + \omega^{2} \mathbf{G}^{\sharp},
\end{equation}
and finally $\mathbf{Z} (\omega) = \left( \mathbf{G} + \omega^{2} \mathbf{G}^{\sharp} \right)^{\sharp}$.
\end{proof}

\section{Proof of Theorem.~\ref{th:sum_lorentzians}}\label{app:proof_sum_lorentzian}

\begin{proof}
Since the process is reversible, $\mathbf{G}$ is diagonalizable and has the spectral 
decomposition $\mathbf{G} = \sum_{k=1}^{n-1} \omega_{k} \boldsymbol\Pi_{k}$.
Therefore, by Prop.~\ref{prop:gen_fund_mat}, given that for all $\omega \in \mathbb{R}^{+} \setminus \{ 0 \}$
the function $x \mapsto x / (x^{2} + \omega^{2})$ is analytic on the domain $\mathbb{R}^{+} \supset \sigma(\mathbf{G})$,
the spectral decomposition of $\mathbf{Z}(\omega)$ is
\begin{equation}
\mathbf{Z}(\omega) = \sum_{k = 1}^{n-1} \frac{\omega_{k}}{\omega_{k}^{2} + \omega^{2}} \boldsymbol\Pi_{k}.
\end{equation}
Because the process is reversible, it is stationary and applying 
Prop.~\ref{prop:quad_form_psd} to the spectral decomposition of $\mathbf{Z}(\omega)$ provides Eq.~(\ref{eq:sum_lorentzians}).

The quadratic form $\left\langle \mathbf{x}, \boldsymbol\Pi_{k} \mathbf{x} \right\rangle _{\boldsymbol\pi}$
is positive semi-definite because $\Pi_{k}$ is a projector and $\sigma \left( \Pi_{k} \right) = \{ 0, 1 \} \subset \mathbb{R}^{+}$,
so we can define $\gamma_{k} \triangleq \left\langle \mathbf{x}, \boldsymbol\Pi_{k} \mathbf{x} \right\rangle^{1/2} _{\boldsymbol\pi}$.
\end{proof}

\section{Proof of Criterion~\ref{prop:psd_scaling}}\label{app:proof_psd_scaling}

We use Def.~3 of \cite{Vigneron2019} as the definition for a function $f(x)$ to scale as $x^{\alpha}$.
In the context we are interested in, it states that a function $f$ is asymptotically homogeneous at infinity 
if there exists $\alpha \in \mathbb{R}$ such that
\begin{equation}
\lim_{x \rightarrow \infty} \log \left( \frac{f(x)}{x^{\alpha}} \right) = 0.
\end{equation}

The result derives from the asymptotic expansion for the power spectral density $S_{X} (\omega)$
(which we prove afterwards Sec.~\ref{app:proof_prop_asymptotic_expansion}):

\begin{proposition}
Assume the scalings $\omega_{k} \sim k^{\alpha}$ with $\alpha > 0$, and $\gamma_{k} \sim k^{\beta}$,
so that $\alpha  > |2 \beta + 1|$, then
\begin{equation}
S_{X} (\omega) \sim \omega^{\zeta} - K \omega^{-2},
\end{equation}
where $\zeta = (2 \beta - \alpha + 1) / \alpha$,
and $K > 0$ if it exists $k \in \mathbb{N}$ so that $k = 2 \beta + \alpha$, else $K=0$.
\label{prop:asymptotic_expansion}
\end{proposition}

Since $-2 < \zeta < 0$ then $\zeta + 2 > 0$ and it immediately follows:
\begin{equation}
\log \left( \frac{S_{X} (\omega)}{\omega^{\zeta}} \right) \sim \log \left( 1 - \frac{K}{\omega^{\zeta + 2}} \right)
\underset{\omega \rightarrow \infty}{\longrightarrow} 0.
\end{equation}
Therefore $S_{X}$ is asymptotically homogeneous at $\omega \rightarrow \infty$ and $S_{X} (\omega)$ scales as $\omega^{\zeta}$.

\subsection{Proof of Prop.~\ref{prop:asymptotic_expansion}}\label{app:proof_prop_asymptotic_expansion}

\begin{proof}
We write the scaling laws as $\omega_{k} = \omega_{0} k^{\alpha}$ and $\gamma_{k} = \gamma_{0} k^{\beta}$,
then $S_{X} (\omega) = \frac{\gamma_{0}^{2}}{\omega_{0}} S_{\alpha \beta} (\bar{\omega})$, where $\bar{\omega} \triangleq \omega / \omega_{0}$ and
$S_{\alpha \beta} (\bar{\omega}) = \sum_{k=1}^{\infty} f_{\alpha \beta}(k)$, with
\begin{equation}
f_{\alpha \beta}(x) = \frac{x^{2\beta + \alpha}}{x^{2 \alpha} + \bar{\omega}^{2}}.
\end{equation}

We start with the Euler-Maclaurin asymptotic expansion to evaluate the sum:
\begin{equation}
\sum_{k=0}^{\infty} f_{\alpha \beta}(k) \sim I_{\alpha \beta}(\bar{\omega}) + \frac{f_{\alpha \beta}(\infty) + f_{\alpha \beta}(0)}{2}
     + \sum_{k=1}^{\infty} \frac{B_{2k}}{(2k)!} \left[ f_{\alpha \beta}^{(2k-1)}(\infty) - f_{\alpha \beta}^{(2k-1)}(0) \right],
\label{eq:Euler-Maclaurin}
\end{equation}
where $I_{\alpha \beta}(\bar{\omega}) \triangleq \int_{0}^{\infty} f_{\alpha \beta}(x) \; dx$, and $B_{k}$ is the $k$th Bernoulli number.

We first look at the behavior of $f_{\alpha \beta}$ and its derivatives as $x \rightarrow \infty$. Since $\alpha > 0$,
\begin{equation}
f_{\alpha \beta}(x) \underset{x \rightarrow \infty}{\sim} x^{2\beta - \alpha},
\end{equation}
and the $k$th derivative satisfies (see \cite{2381258}):
\begin{equation}
f_{\alpha \beta}^{(k)}(x) \underset{x \rightarrow \infty}{\sim} 
    \frac{\Gamma(2\beta + \alpha)}{\Gamma(2\beta + \alpha - k)} x^{2\beta - \alpha - k},
\end{equation}
where $\Gamma$ is the Euler gamma function.
Because for all $k \in \mathbb{N}$, $2 \beta - k < |2\beta +1| < \alpha$,
we have $f_{\alpha \beta}^{(k)}(\infty) \triangleq \lim_{x \rightarrow \infty} f_{\alpha \beta}^{(k)}(x) = 0$.

Moreover, because $\sum_{k=1}^{\infty} f_{\alpha \beta}(k) = \sum_{k=0}^{\infty} f_{\alpha \beta}(k) - f_{\alpha \beta}(0)$,
the Euler-Maclaurin expansion Eq.~(\ref{eq:Euler-Maclaurin}) implies
\begin{equation}
S_{\alpha \beta} (\bar{\omega}) \sim I_{\alpha \beta}(\bar{\omega}) - R_{\alpha \beta} (\bar{\omega}),
\end{equation}
where
\begin{equation}
R_{\alpha \beta} (\bar{\omega}) = \frac{1}{2} f_{\alpha \beta}(0) + \sum_{k=1}^{\infty} \frac{B_{2k}}{(2k)!} f_{\alpha \beta}^{(2k-1)}(0).
\end{equation}

Therefore, by Lemma~\ref{lemma:int_fab} and Lemma~\ref{lemma:Rab}:
\begin{equation}
S_{\alpha \beta} (\bar{\omega})
    = \frac{\pi}{2 \alpha} \sec \left( \pi \frac{1+2\beta}{2 \alpha} \right) \bar{\omega}^{\frac{2 \beta - \alpha +1}{\alpha}} - K \bar{\omega}^{-2}.
\end{equation}
\end{proof}

\begin{remark}
The power spectral density $S_{\alpha \beta} (\bar{\omega})$ is finite at zero frequency ($\bar{\omega} = 0$), because
\begin{equation}
S_{\alpha \beta} (0) = \sum_{k=1}^{\infty} k^{2 \beta - \alpha} = \zeta (\alpha - 2 \beta),
\end{equation}
where $\zeta$ is the Riemann zeta function.
The serie converges if $\alpha - 2 \beta > 1$, which is satisfied because $\alpha > |2 \beta + 1| \geq 2 \beta + 1$. 
\end{remark}

\begin{lemma}
If $\alpha > 0$ and $\alpha > |2 \beta + 1|$ then
\begin{equation}
I_{\alpha \beta}(\bar{\omega}) = \frac{\pi}{2 \alpha} \sec \left( \pi \frac{1+2\beta}{2 \alpha} \right) \bar{\omega}^{\frac{2 \beta - \alpha +1}{\alpha}},
\end{equation}
where $\sec x \triangleq 1 / \cos x$.
\label{lemma:int_fab}
\end{lemma}

\begin{proof}
We use the fact that the fractional absolute moments of the Cauchy distribution:
\begin{equation}
p(x) = \frac{\gamma}{\pi} \frac{1}{x^{2} + \gamma^{2}},
\end{equation}
are, for $|p| < 1$, $\left\langle |x|^{p} \right\rangle = \gamma^{p} \sec ( \pi p / 2 )$.
In other words,
\begin{equation}
\int_{0}^{\infty} \frac{x^{p}}{x^{2} + \gamma^{2}} \; dx = \frac{\pi}{2} \gamma^{p - 1} \sec \left( \frac{\pi p}{2} \right).
\label{eq:p-moment-cauchy}
\end{equation}

Because $\alpha > 0$, we can write:
\begin{equation}
I_{\alpha \beta}(\bar{\omega}) = \int_{0}^{\infty} \frac{x^{2\beta + \alpha}}{x^{2 \alpha} + \bar{\omega}^{2}} \; dx
= \int_{0}^{\infty} \frac{\left( x^{\alpha} \right)^{\frac{2\beta + \alpha + 1}{\alpha}}}{\left( x^{\alpha} \right)^{2} + \bar{\omega}^{2}} \; \frac{dx}{x},
\end{equation}
from the change of variable $y = x^{\alpha}$, we have $dy / y = \alpha dx / x$
and using Eq.~(\ref{eq:p-moment-cauchy}) with $p = (2 \beta + 1) / \alpha$ it results
\begin{equation}
I_{\alpha \beta}(\bar{\omega}) = \frac{1}{\alpha} \int_{0}^{\infty} \frac{y^{\frac{2\beta + 1}{\alpha}}}{y^{2} + \bar{\omega}^{2}} \; dy
= \frac{\pi}{2 \alpha } \bar{\omega}^{\frac{2\beta + 1}{\alpha} - 1} \sec \left( \pi \frac{1+2\beta}{2 \alpha} \right),
\end{equation}
which is defined because the condition $\alpha > |2 \beta + 1|$ implies $|p| < 1$.
\end{proof}

\begin{remark}
This proof exhibits a relationship between the power spectral density of a Markov chain with power law scaling eigenstructure
and the fractional absolute moments of the Cauchy distribution, basically $S(\omega) \sim \left\langle |x|^{1 + \zeta} \right\rangle \sim \omega^{\zeta}$.

Note that the fractional absolute moments are closely related to the Mellin transform which is often
"Used in place of Fourier's transform when scale invariance is more relevant than shift invariance" \cite{Bertrand1995}.
It would therefore be interesting to study the potential relation between $1/f$ power spectral densities and the Mellin transform.
Maybe one could derive a variant of the Wiener-Kinchine theorem based on the Mellin transform instead of the Fourier transform ?
\end{remark}

\begin{lemma}
If $\alpha > 0$ then $R_{\alpha \beta}(\bar{\omega}) = K \bar{\omega}^{-2}$, where
\begin{equation}
K =
\begin{cases}
    1 / 2 & \text{if } 2\beta + \alpha = 0, \\
    \frac{B_{2p}}{(2p)!} \Gamma(2\beta + \alpha) & \text{if it exists } p \in \mathbb{N} \setminus \{0\} \text{ so that } 2p = 2\beta + \alpha + 1, \\
    0 & \text{otherwize}.
\end{cases}
\end{equation}
That is $K > 0$ if it exists $k \in \mathbb{N}$ so that $k = 2 \beta + \alpha$, else $K=0$.
\label{lemma:Rab}
\end{lemma}

\begin{proof}
Since $\alpha > 0$,
\begin{equation}
f_{\alpha \beta}(x) \underset{x \rightarrow 0}{\sim} \frac{x^{2\beta + \alpha}}{\bar{\omega}^{2}},
\end{equation}
and $k$th derivative ($k \in \mathbb{N}$) satisfies:
\begin{equation}
f_{\alpha \beta}^{(k)}(x) \underset{x \rightarrow 0}{\sim} 
    \frac{\Gamma(2\beta + \alpha)}{\Gamma(2\beta + \alpha - k)} \frac{x^{2\beta + \alpha - k}}{\bar{\omega}^{2}}.
\end{equation}
Therefore, for all $k \in \mathbb{N}$,
\begin{equation}
f_{\alpha \beta}^{(k)}(0) =
\begin{cases}
    \Gamma(2\beta + \alpha) / \bar{\omega}^{2} & \text{if } k = 2\beta + \alpha, \\
    0 & \text{otherwize}.
\end{cases}
\end{equation}
As a consequence,
\begin{equation}
R_{\alpha \beta}(\bar{\omega}) =
\begin{cases}
    1 / (2 \bar{\omega}^{2}) & \text{if } 2\beta + \alpha = 0, \\
    \frac{B_{2p}}{(2p)!} \frac{\Gamma(2\beta + \alpha)}{\bar{\omega}^{2}} & \text{if it exists } p \in \mathbb{N} \setminus \{0\} \text{ so that } 2p = 2\beta + \alpha + 1, \\
    0 & \text{otherwize}.
\end{cases}
\end{equation}

\end{proof}

\section{PSD of Birth-death processes in the heavy traffic regime}\label{app:birth_deat_heavy_traffic}

A birth-death process is a Markov process where from a given state one can only realize a birth transition increasing the state count,
or a death transition reducing the state count (except from the first state from which on ly births are allowed).
The infinitesimal generator of this process is a tridiagonal matrix:
\begin{equation}
\mathbf{G} =
\begin{pmatrix}
& \lambda_{0} & -\lambda_{0}           &                      &              &        \\
& -\mu_{1}    & \mu_{1} + \lambda_{1} & -\lambda_{1}          &              &        \\
&             & -\mu_{2}              & \mu_{2} + \lambda_{2} & -\lambda_{2} &        \\
&             &                       & \ddots                & \ddots       & \ddots \\
\end{pmatrix},
\end{equation}
where $\lambda_{n}$ are the birth rates and $\mu_{n}$ the death rates.
In this appendix, we provide a general formula for the power spectral density of a birth-death process in the heavy traffic regime.
More precisely, we see that the coupling coefficient is related to the family of orthogonal
polynomials associated to the characteristic polynomial of the generator.

\subsection{Characteristic polynomial of the infinitesimal generator}

Because the matrix is tridiagonal, the characteristic polynomial $f_{n}$ of the truncated $n \times n$ generator $\mathbf{G}_{n}$
satisfies the three terms recurrence relation:
\begin{equation}
x f_{n}(x) = - \lambda_{n} \mu_{n} f_{n-1}(x) + \left( \mu_{n+1} + \lambda_{n+1} \right) f_{n}(x) - f_{n+1}(x),
\end{equation}
with initial conditions $f_{0}(x) = 1$ and $f_{1}(x) = \lambda_{1} - x$. Since $\lambda_{n} \mu_{n} > 0$, 
from Favard theorem, the polynomial $f_{n}$ is \textit{orthogonal} \cite{Koornwinder2013}.
Therefore the eigenvalues of $\mathbf{G}$ are roots of orthogonal polynomials.

Moreover the roots of successive polynomials $f_{n}$ and $f_{n+1}$ from a family of orthogonal polynomials $\{ f_{n} \}_{n \in \mathbb{N}}$,
which are the eigenvalues of $\mathbf{G}_{n}$ and $\mathbf{G}_{n+1}$, are interleaved \cite{Koornwinder2013}.
In the limit $n \rightarrow \infty$, the eigenfrequencies of $\mathbf{G}$ are therefore densely covering a wide span of values,
which is favorable to the emergence of $1/f$ scaling.



\subsection{Heavy traffic limit}

We define the server utilization $\rho_{i} \triangleq \lambda_{i - 1} / \mu_{i}$.
In the heavy traffic approximation $\rho_{i} \sim 1$, for all $i$, and $\mathbf{G}$ is symmetric.

In \cite{Molinari2019} it is shown that the coefficients of the eigenvector of a tridiagonal matrix are related
to the family of orthogonal characteristic polynomials $\left\{ f_{i} \right\}_{1 \leq i \leq n}$.
More specifically, for symmetric tridiagonal matrix the eigenvector $\mathbf{v}_{k} = \left( v_{k, i} \right)_{1 \leq i \leq n}$
associated to the eigenvalue $\omega_{k}$ is:
\begin{equation}
v_{k, i} = q_{k} \psi_{i} \left( \omega_{k} \right),
\end{equation}
where $\psi_{i} (x) \triangleq f_{i-1}(x) / \beta_{i-1}$ and $\beta_{i} \triangleq \prod_{l=1}^{i} \mu_{l}$.
The normalization coefficient is:
\begin{equation}
q_{k}^{2} = \frac{\beta_{n - 1}}{\psi_{n}\left( \omega_{k} \right) \prod_{i \neq k} \left( \omega_{k} - \omega_{i} \right)}.
\end{equation}

The coupling factor is therefore:
\begin{equation}
\gamma_{k}^{2} = q_{k}^{2} \left( \sum_{i=1}^{n} \pi_{i} x_{i} \psi_{i} \left( \omega_{k} \right) \right)
                       \left( \sum_{i=1}^{n} x_{i} \psi_{i} \left( \omega_{k} \right) \right),
\end{equation}

The steady-state population for the state $i$ of the birth-death process is $\pi_{i} = \pi_{0} \prod_{l=1}^{k} \rho_{i}$,
so in the heavy traffic limit the populations are almost equidistributed ($\pi_{i} \sim \pi_{0}$ for all $i$).
Therefore,
\begin{equation}
\gamma_{k} \simeq \sqrt{\pi_{0}} q_{k} \hat{\mathbf{x}} \left( \omega_{k} \right),
\end{equation}
where we introduced the generalized Fourier serie with respect to the family of functions $\left\{ \psi_{i} \right\}_{1 \leq i \leq n}$:
\begin{equation}
\hat{\mathbf{x}} (\omega) \triangleq \sum_{i=1}^{n} x_{i} \psi_{i} \left( \omega \right).
\end{equation}
We see in App.~\ref{app:graph_fourier_transform} that, when the generator is symmetric, the factor $\gamma_{k}$ is always a Fourier serie,
more specifically it is the graph Fourier transform of the Markov chain graph.

\section{Relationship between the PSD and the graph Fourier transform}\label{app:graph_fourier_transform}

A graph can be associated to the Markov where each vertice is a state of the chain and the edges of the graph are defined by the transition rates.
When the generator $\mathbf{G}$ is symmetric, it is also the Laplacian matrix of the graph associated to the Markov chain.
Let $\omega_{k}$ be an eigenvalue of $\mathbf{G}$ and $\mathbf{v}_{k} = (v_{k,i})_{1 \leq i \leq n}$ the corresponding eigenvector,
the \textit{graph Fourier transform} \cite{Ricaud2019} $\hat{\mathbf{x}}$ of $\mathbf{x} \in \mathbb{R}^{n}$
at $\omega_{k}$ is defined as
\begin{equation}
\hat{\mathbf{x}} \left( \omega_{k} \right) \triangleq \left\langle \mathbf{x}, \mathbf{v}_{k} \right\rangle = \sum_{i=1}^{n} x_{i} v_{k,i}.
\label{eq:gft_def}
\end{equation}

The coupling factor $\gamma_{k}$ satisfies
\begin{equation}
\gamma_{k}^{2} = \left\langle \mathbf{x}, \mathbf{v}^{\top}_{k} \mathbf{v}_{k} \mathbf{x} \right\rangle _{\boldsymbol\pi}
= \left( \sum_{i=1}^{n} \pi_{i} x_{i} v_{k,i} \right) \left( \sum_{i=1}^{n} x_{i} v_{k,i} \right).
\end{equation}

Since the generator is symmetric, the transition matrix $\mathbf{P}$ is a doubly stochastic matrix and the Markov process as an uniform
limiting distribution: $\pi_{i} = 1/n$ for all $i$. Therefore, the coupling factor can be written using the graph Fourier transform
Eq.~(\ref{eq:gft_def}):
\begin{equation}
\gamma_{k} = \frac{1}{\sqrt{n}} \sum_{i=1}^{n} x_{i} v_{k,i} = \frac{1}{\sqrt{n}} \hat{\mathbf{x}} \left( \omega_{k} \right).
\end{equation}
We can define the \textit{graph power spectral density} as
\begin{equation}
S_{X} (\omega) = \frac{1}{n} \sum_{k=1}^{n} \hat{\mathbf{x}} \left( \omega_{k} \right)^{2} \frac{\omega_{k}}{\omega_{k}^{2} + \omega^{2}}.
\end{equation}

\bibliographystyle{amsplain}
\bibliography{biblio}

\end{document}